\documentclass{amsart}
\usepackage{amsmath, amssymb, amsthm}
\usepackage{color}
\usepackage{hyperref}
\usepackage[shortlabels]{enumitem}
\usepackage[table,dvipsnames]{xcolor}
\usepackage{diagbox}
\usepackage{nicefrac} %

\newcommand{\defn}[1]{\emph{\color{blue} #1}}
\newcommand{\floor}[1]{\left\lfloor {#1} \right\rfloor}
\newcommand{\ceil}[1]{\left\lceil {#1} \right\rceil}
\newcommand{\sprod}[2]{\langle {#1} , {#2} \rangle} 
\newcommand{\set}[2]{\ensuremath{\left\{#1\,\middle|\,#2\right\}}}

\newcommand{\Bset}[2]{\ensuremath{\Big\{#1\,\Big|\,#2\Big\}}}
\newcommand{\uno}{\mathbf{1}}

\newtheorem{theorem}{Theorem}[section]
\newtheorem{lemma}[theorem]{Lemma}
\newtheorem{cor}[theorem]{Corollary}
\newtheorem{proposition}{Proposition}[section]

\theoremstyle{remark}
\newtheorem{example}{Example}[section]
\newtheorem{remark}[theorem]{Remark}

\theoremstyle{definition}
\newtheorem{definition}{Definition}[section]

\def\RR{\mathbb{R}}
\def\NN{\mathbb{N}}
\def\bfe{\mathbf{e}}
\def\bfp{\mathbf{p}}
\def\bfq{\mathbf{q}}
\def\bfv{\mathbf{v}}

\def\bfn{\mathbf{n}}
\def\bfm{\mathbf{m}}

\newcommand{\ssm}{\smallsetminus}

\newcommand{\Knk}{\Kuh{n}{k}}%
\newcommand{\Kuh}[2]{K_{#1}^{(#2)}}%
\newcommand{\Kumh}[3]{\Kuh{#1 \times #2}{#3}}%

\newcommand{\Hnk}{\Hs{n}{k}}%
\newcommand{\Hs}[2]{\Delta_{#1,#2}}%

\newcommand{\cd}{\mathsf{cd}}%
\newcommand{\scd}{\mathsf{scd}}%

\newcommand{\PC}[1]{S_{#1}}%
\newcommand{\Ppi}{\PC{\pi}}

\newcommand{\hep}[1]{\mathcal P({#1})}%

\DeclareMathOperator{\conv}{conv}
\DeclareMathOperator{\intr}{int}
\DeclareMathOperator{\relint}{relint}
\DeclareMathOperator{\pyr}{pyr}

\usepackage{todonotes}

\title[The hypersimplicial Van Kampen-Flores Theorem]{The convex dimension of %
hypergraphs and the hypersimplicial Van Kampen-Flores Theorem}

\author{Leonardo Mart\'inez-Sandoval}
\address[L. Mart\'inez-Sandoval]{Sorbonne Universit\'e, Institut de Math\'ematiques de Jussieu - Paris Rive Gauche (UMR 7586), Paris, France and Faculty of Sciences, National Autonomous University of Mexico, Mexico City, Mexico}
\email[L.~Mart\'inez-Sandoval]{leomtz@ciencias.unam.mx}
\thanks{Research supported by the grant ANR-17-CE40-0018 of the French National Research Agency ANR (project CAPPS)}

\author{Arnau Padrol}
\address[A. Padrol]{Sorbonne Universit\'e, Institut de Math\'ematiques de Jussieu - Paris Rive Gauche (UMR 7586), Paris, France}
\email[A.~Padrol]{arnau.padrol@imj-prg.fr}

\begin{document}

\begin{abstract}

    The convex dimension of a $k$-uniform hypergraph is the smallest dimension~$d$ for which there is an injective mapping of its vertices into~$\RR^d$ such that the set of $k$-barycenters of all hyperedges is in convex position.

    We completely determine the convex dimension of complete $k$-uniform hypergraphs, which settles an open question by Halman, Onn and Rothblum, who solved the problem for complete graphs. We also provide lower and upper bounds for the extremal problem of estimating the maximal number of hyperedges of $k$-uniform hypergraphs on $n$ vertices with convex dimension~$d$. 
    
    To prove these results, we restate them in terms of affine projections that preserve the vertices of the hypersimplex. More generally, we provide a full characterization of the projections that preserve its $i$-dimensional skeleton. In particular, we obtain a hypersimplicial generalization of the linear van Kampen-Flores theorem: 
    for each $n$, $k$ and $i$ we determine onto which dimensions can the $(n,k)$-hypersimplex be linearly projected while preserving its $i$-skeleton.
    
    Our results have direct interpretations in terms of $k$-sets and $(i,j)$-partitions, and are closely related to the problem of finding large convexly independent subsets in Minkowski sums of $k$ point sets.
\end{abstract}

\maketitle

\section{Introduction}

Motivated by problems in convex combinatorial optimization~\cite{OnnRothblum2004}, Halman, Onn and Rothblum introduced the concept of convex dimension of uniform hyper\-graphs \cite{HalmanOnnRothblum2007}.  A \defn{$k$-uniform hypergraph} is a pair~$H=(V,E)$ with~$E\subseteq \binom{V}{k}$; a \defn{convex embedding} of~$H$ into~$\RR^d$
is an injective map $f:V\to \RR^d$ such that
the 
\defn{$k$-barycenters} (i.e.\ the points of the form $\frac{1}{k}\sum_{v\in e}f(v)$ indexed by the hyperedges $e\in E$)
 are in convex position (i.e.\ no point is a convex combination of the others);
and the \defn{convex dimension} of~$H$, denoted~\defn{$\cd(H)$}, is the minimal~$d$ for which a convex embedding of~$H$ into~$\RR^d$ exists. 

Their article focused on graphs, the $k=2$ case. They studied the problem of determining the convex dimension for specific families of graphs: paths, cycles, complete graphs and bipartite graphs. They also investigated the extremal problem of determining the maximum number of edges that a graph on $n$ vertices and fixed convex dimension can have. The latter problem has been studied afterwards by several authors, in particular because of its strong relation with the problem of determining large convex subsets in Minkowski sums~\cite{Bilka2010,EisenbrandPachRothvossSopher2008}. Indeed, convex embeddings of subhypergraphs of complete $k$-partite $k$-uniform hypergraphs correspond to subsets in convex position inside the Minkowski sum of $k$ sets of points. Diverse variants of the case $k=2$ have been considered in the plane~\cite{Bilka2010,EisenbrandPachRothvossSopher2008,GarciaMarcoKnauer2016,SkomraThomasse,Tiwary2014}, and also in~$\RR^3$~\cite{SwanepoelValtr2010}.

For $k>2$, the only result of which we are aware of is the upper bound $\cd(H)\leq 2k$ for any $k$-uniform hypergraph~$H$, proved by Halman et al.\ by mapping the vertices onto points on the moment curve in $\RR^{2k}$~\cite{HalmanOnnRothblum2007}. 

Our first result is the complete determination of the convex dimension of~$\Knk:=([n],\binom{[n]}{k})$, the \defn{complete $k$-uniform hypergraph} on~$n$ vertices, for any $k$, $1\leq k \leq n-1$.

\begin{theorem}
\label{thm:main}
Given positive integers $n$ and $k$ such that $2\leq k\leq n-2$, we have that 
	\[
	\cd(K_n^{(k)})=
		\begin{cases}
			2k &\text{if $n\geq 2k+2$,}\\
			n-2 &\text{if $n\in \{2k-1,2k,2k+1\}$,}\\
			2n-2k &\text{if $n\leq 2k-2$.}
		\end{cases}
	\]

Also, $\cd(K_2^{(1)})=1$ and $\cd(K_n^{(1)})=\cd(K_n^{(n-1)})=2$ for $n\geq 3$.
\end{theorem}
This matches and extends the results for $k=2$ in \cite{HalmanOnnRothblum2007}, where it is proved that $\cd(K_n)=4$ for $n\geq 6$. Table \ref{tab:values} shows the explicit values of $\cd(K_n^{(k)})$ given by Theorem~\ref{thm:main} for small values of $n$ and $k$. Note the vertical symmetry in each column.

\begin{table}[htpb]
	\centering
	\begin{tabular}{|c|ccccccccccccccccc|}
		\hline
		$k\setminus n$ &                 2  &                 3  &                 4  &                 5  &                 6  &                 7  &                 8  &                 9  &                 10 &                 11 &                  12 &                  13 &                  14 &                  15 &                  16 &                  17 &                  18\\
		\hline
		1 &  \cellcolor{green}1 &  \cellcolor{green}2 &  \cellcolor{pink}2 &  \cellcolor{pink}2 &  \cellcolor{pink}2 &  \cellcolor{pink}2 &  \cellcolor{pink}2 &  \cellcolor{pink}2 &  \cellcolor{pink}2 &  \cellcolor{pink}2 &   \cellcolor{pink}2 &   \cellcolor{pink}2 &   \cellcolor{pink}2 &   \cellcolor{pink}2 &   \cellcolor{pink}2 &   \cellcolor{pink}2 &   \cellcolor{pink}2 \\
		2 &                    &  \cellcolor{green}2 &\cellcolor{yellow}2 &\cellcolor{yellow}3 &  \cellcolor{pink}4 &  \cellcolor{pink}4 &  \cellcolor{pink}4 &  \cellcolor{pink}4 &  \cellcolor{pink}4 &  \cellcolor{pink}4 &   \cellcolor{pink}4 &   \cellcolor{pink}4 &   \cellcolor{pink}4 &   \cellcolor{pink}4 &   \cellcolor{pink}4 &   \cellcolor{pink}4 &   \cellcolor{pink}4\\
		3 &                    &                    &  \cellcolor{pink}2 &\cellcolor{yellow}3 &\cellcolor{yellow}4 &\cellcolor{yellow}5 &  \cellcolor{pink}6 &  \cellcolor{pink}6 &  \cellcolor{pink}6 &  \cellcolor{pink}6 &   \cellcolor{pink}6 &   \cellcolor{pink}6 &   \cellcolor{pink}6 &   \cellcolor{pink}6 &   \cellcolor{pink}6 &   \cellcolor{pink}6 &   \cellcolor{pink}6 \\
		4 &                    &                    &                    &  \cellcolor{pink}2 &  \cellcolor{pink}4 &\cellcolor{yellow}5 &\cellcolor{yellow}6 &\cellcolor{yellow}7 &  \cellcolor{pink}8 &  \cellcolor{pink}8 &   \cellcolor{pink}8 &   \cellcolor{pink}8 &   \cellcolor{pink}8 &   \cellcolor{pink}8 &   \cellcolor{pink}8 &   \cellcolor{pink}8 &   \cellcolor{pink}8\\
		5 &                    &                    &                    &                    &  \cellcolor{pink}2 &  \cellcolor{pink}4 &  \cellcolor{pink}6 &\cellcolor{yellow}7 &\cellcolor{yellow}8 &\cellcolor{yellow}9 &  \cellcolor{pink}10 &  \cellcolor{pink}10 &  \cellcolor{pink}10 &  \cellcolor{pink}10 &  \cellcolor{pink}10 &  \cellcolor{pink}10 &  \cellcolor{pink}10 \\
		6 &                    &                    &                    &                    &                    &  \cellcolor{pink}2 &  \cellcolor{pink}4 &  \cellcolor{pink}6 &  \cellcolor{pink}8 &\cellcolor{yellow}9 &\cellcolor{yellow}10 &\cellcolor{yellow}11 &  \cellcolor{pink}12 &  \cellcolor{pink}12 &  \cellcolor{pink}12 &  \cellcolor{pink}12 &  \cellcolor{pink}12\\
		7 &                    &                    &                    &                    &                    &                    &  \cellcolor{pink}2 &  \cellcolor{pink}4 &  \cellcolor{pink}6 &  \cellcolor{pink}8 &  \cellcolor{pink}10 &\cellcolor{yellow}11 &\cellcolor{yellow}12 &\cellcolor{yellow}13 &  \cellcolor{pink}14 &  \cellcolor{pink}14 &  \cellcolor{pink}14 \\
		8 &                    &                    &                    &                    &                    &                    &                    &  \cellcolor{pink}2 &  \cellcolor{pink}4 &  \cellcolor{pink}6 &   \cellcolor{pink}8 &  \cellcolor{pink}10 &  \cellcolor{pink}12 &\cellcolor{yellow}13 &\cellcolor{yellow}14 &\cellcolor{yellow}15 &  \cellcolor{pink}16 \\
		9 &                    &                    &                    &                    &                    &                    &                    &                    &  \cellcolor{pink}2 &  \cellcolor{pink}4 &  \cellcolor{pink}6  &   \cellcolor{pink}8 &  \cellcolor{pink}10 &  \cellcolor{pink}12 &  \cellcolor{pink}14 &\cellcolor{yellow}15 &\cellcolor{yellow}16  \\
		\hline
		\end{tabular}
		 \bigskip

            \caption{First values of $\cd(K_n^{(k)})$. Green values correspond to exceptional cases with small values of $n$ and $k$. Yellow values correspond to the cases $n\in\{2k-1,2k,2k+1\}$, when $k\geq 2$. Red values correspond to the cases $n\geq 2k+2$ or $n\leq 2k-2$.}\label{tab:values}
\end{table}

In a different context, the convex-hull of all $k$-barycenters of a point-set~$S$ has also been studied under the name of \defn{$k$-set polytope}~and denoted by \defn{$P_k(S)$} \cite{AndrzejakWelzl2003,EdelsbrunnerValtrWelzl1997} in relation to the study of $k$-sets, $j$-facets and $(i,j)$-partitions~\cite{AndrzejakWelzl2003, Wagner2008}. Indeed, vertices of $P_k(S)$ correspond bijectively to \defn{$k$-sets} of~$S$; that is, to $k$-element subsets that can be separated by a hyperplane (higher-dimensional faces of $P_k(S)$ have similar interpretations, which we explain in Section~\ref{sec:ksets}). In this language, Theorem~\ref{thm:main} determines for which dimensions we can find point sets of cardinality~$n$ for which all $k$-barycenters are vertices of the $k$-set polytope; that is, point-sets for which every $k$-element subset can be separated with a linear hyperplane.

We provide a polyhedral proof of Theorem~\ref{thm:main}. Namely, we reformulate the existence of a convex embedding of $\Knk$ into $\RR^d$ in terms of affine projections that strictly preserve the vertices of the \defn{hypersimplex $\Hnk$}, that is, the polytope whose vertices are the $\binom{n}{k}$ incidence vectors of $k$-subsets of $[n]$. This polyhedral formulation is closer to the original set-up of convex combinatorial optimization~\cite{OnnRothblum2004}.

Hypersimplices are a widely studied family of polytopes that arise naturally in very diverse contexts, and there has been a recent interest on hypersimplex projections~\cite{OlarteSantos, Postnikov2019} motivated by a result of Galashin~\cite{Galashin2018} who showed that certain subdivisions induced by hypersimplex projections are in bijection with reduced plabic graphs~\cite{OhPostnikovSpeyer2015}, used to describe the stratification of the totally nonnegative Grassmannian~\cite{Postnikov06}. 

As we shall see, Theorem~\ref{thm:main} is a particular case of a more general result. Let \defn{$d=d(n,k,i)$} be the smallest dimension for which we can find a projection $\pi:\Hnk\to \RR^d$ that strictly preserves the $i$-dimensional skeleton of $\Hnk$. We determine the values of $d(n,k,i)$ in Theorem~\ref{thm:skeleton}, which is proved in
Section~\ref{sec:hypersimplicial-neighborly}
with the framework used by Sanyal when studying the number of vertices of Minkowski sums~\cite{Sanyal2009}, based on Ziegler's projection lemma~\cite{Ziegler2004}.

\begin{theorem}
	\label{thm:skeleton}
	Given positive integers $n$, $k$, $i$ such that $1\leq k\leq n-1$ and $0\leq i\leq n-1$, the value of $d(n,k,i)$ is determined as follows.
		\[
		d(n,k,i)=
			\begin{cases}
				2k+2i &\text{if $n\geq 2k+2i+2$},\\
				2n-2k+2i &\text{if $n\leq 2k-2i-2$},\\
                                n-1 &\text{if $2k-2i-1\leq n \leq 2k+2i+1$, $k\in A_{n,i}$},\\
				n-2 &\text{if $2k-2i-1\leq n \leq 2k+2i+1$, $k\notin A_{n,i}$}.\\
			\end{cases}
		\]
	Where
	\begin{align*}
		A_{n,i} &= \{1,2,\ldots,i+1\}\cup \{n-i-1,n-i,\ldots,n-1\}.\\
	\end{align*}
\end{theorem}

For $k=1$, we get $d(n,1,i)=2i+2$ for $n\geq 2i+4$ and $d(n,1,i)=n-1$ otherwise, which is a reformulation of the classical corollary of Radon's theorem that states that the simplex is the only $d$-dimensional polytope (from now on abbreviated as \defn{$d$-polytope}) that is more than $\floor{\frac{d}{2}}$-neighborly~\cite[Thm.~7.1.4]{Grunbaum}; or equivalently, that no linear projection of a $(2i+2)$-simplex onto $\RR^{2i+1}$ preserves its $i$-skeleton. This result is sometimes referred to as the \defn{linear van Kampen-Flores Theorem} (e.g.\ in \cite[p.~95]{RoerigSanyal2012}); thus Theorem~\ref{thm:skeleton} could be called the \defn{hypersimplicial linear van Kampen-Flores Theorem}.

Prodsimplicial linear van Kampen-Flores Theorems for products of simplices were proved by Matschke, Pfeifle and Pilaud~\cite{Matschke2011} and R\"{o}rig and Sanyal~\cite{RoerigSanyal2012}. One of their motivations was the study of dimensional ambiguity.
Gr\"unbaum defined a polytope~$P$ to be \defn{dimensionally $i$-ambiguous} if its $i$-skeleton is isomorphic to that of a polytope~$Q$ of a different dimension~\cite[Ch.~12]{Grunbaum}. 
Few polytopes are known to be dimensionally ambiguous. Examples include: simplices via \defn{neighborly polytopes}, cubes via \defn{neighborly cubical polytopes}~\cite{JoswigZiegler2000}, products of polygons via \defn{projected products of polygons}~\cite{SanyalZiegler2010,Ziegler2004}, and the aforementioned products of simplices via \defn{prodsimplicial-neighborly polytopes}~\cite{Matschke2011}. Our results show that the $(n,k)$-hypersimplex is dimensionally $i$-ambiguous except, maybe, when $k\in A_{n,i}$ and $2k-2i-1\leq n\leq 2k+2i+1$, which happens only for small values of~$n$ and~$k$.

For convenience, we present the case $i=2$ in Table~\ref{tab:i2values}, and we invite the reader to compare it with Table~\ref{tab:values}. Note once more the vertical symmetry in each column. It directly follows from Theorem \ref{thm:skeleton}, and we emphasize it in Corollary \ref{cor:symmetry}.

\begin{table}[htpb]
	\centering

	\begin{tabular}{|c|cccccccccccccccc|}
		\hline
	$k\setminus n$
	&                 3  &                 4  &                 5  &                 6  &                 7  &                 8  &                 9  &                 10 &                  11 &                  12 &                  13 &                  14 &                  15 &                  16 &                  17 &                  18 \\
		\hline
		1 &  \cellcolor{green}2 &  \cellcolor{green}3 &  \cellcolor{green}4 &  \cellcolor{green}5 &  \cellcolor{green}6 &  \cellcolor{pink}6 &  \cellcolor{pink}6 &  \cellcolor{pink}6 &   \cellcolor{pink}6 &   \cellcolor{pink}6 &   \cellcolor{pink}6 &   \cellcolor{pink}6 &   \cellcolor{pink}6 &   \cellcolor{pink}6 &   \cellcolor{pink}6 &   \cellcolor{pink}6 \\
		2 &  \cellcolor{green}2 &  \cellcolor{green}3 &  \cellcolor{green}4 &  \cellcolor{green}5 &  \cellcolor{green}6 &  \cellcolor{green}7 &  \cellcolor{green}8 &  \cellcolor{pink}8 &   \cellcolor{pink}8 &   \cellcolor{pink}8 &   \cellcolor{pink}8 &   \cellcolor{pink}8 &   \cellcolor{pink}8 &   \cellcolor{pink}8 &   \cellcolor{pink}8 &   \cellcolor{pink}8 \\
		3 &                    &  \cellcolor{green}3 &  \cellcolor{green}4 &  \cellcolor{green}5 &  \cellcolor{green}6 &  \cellcolor{green}7 &  \cellcolor{green}8 &  \cellcolor{green}9 &  \cellcolor{green}10 &  \cellcolor{pink}10 &  \cellcolor{pink}10 &  \cellcolor{pink}10 &  \cellcolor{pink}10 &  \cellcolor{pink}10 &  \cellcolor{pink}10 &  \cellcolor{pink}10 \\
		4 &                    &                    &  \cellcolor{green}4 &  \cellcolor{green}5 &  \cellcolor{green}6 &\cellcolor{yellow}6 &\cellcolor{yellow}7 &\cellcolor{yellow}8 &\cellcolor{yellow} 9 &\cellcolor{yellow}10 &\cellcolor{yellow}11 &  \cellcolor{pink}12 &  \cellcolor{pink}12 &  \cellcolor{pink}12 &  \cellcolor{pink}12 &  \cellcolor{pink}12 \\
		5 &                    &                    &                    &  \cellcolor{green}5 &  \cellcolor{green}6 &  \cellcolor{green}7 &\cellcolor{yellow}7 &\cellcolor{yellow}8 &\cellcolor{yellow} 9 &\cellcolor{yellow}10 &\cellcolor{yellow}11 &\cellcolor{yellow}12 &\cellcolor{yellow}13 &  \cellcolor{pink}14 &  \cellcolor{pink}14 &  \cellcolor{pink}14 \\
		6 &                    &                    &                    &                    &  \cellcolor{green}6 &  \cellcolor{green}7 &  \cellcolor{green}8 &\cellcolor{yellow}8 &\cellcolor{yellow} 9 &\cellcolor{yellow}10 &\cellcolor{yellow}11 &\cellcolor{yellow}12 &\cellcolor{yellow}13 &\cellcolor{yellow}14 &\cellcolor{yellow}15 &  \cellcolor{pink}16 \\
		7 &                    &                    &                    &                    &                    &  \cellcolor{pink}6 &  \cellcolor{green}8 &  \cellcolor{green}9 &\cellcolor{yellow} 9 &\cellcolor{yellow}10 &\cellcolor{yellow}11 &\cellcolor{yellow}12 &\cellcolor{yellow}13 &\cellcolor{yellow}14 &\cellcolor{yellow}15 &\cellcolor{yellow}16 \\
		8 &                    &                    &                    &                    &                    &                    &  \cellcolor{pink}6 &  \cellcolor{pink}8 &  \cellcolor{green}10 &\cellcolor{yellow}10 &\cellcolor{yellow}11 &\cellcolor{yellow}12 &\cellcolor{yellow}13 &\cellcolor{yellow}14 &\cellcolor{yellow}15 &\cellcolor{yellow}16 \\
		9 &                    &                    &                    &                    &                    &                    &                    &  \cellcolor{pink}6 &   \cellcolor{pink}8 &  \cellcolor{pink}10 &\cellcolor{yellow}11 &\cellcolor{yellow}12 &\cellcolor{yellow}13 &\cellcolor{yellow}14 &\cellcolor{yellow}15 &\cellcolor{yellow}16 \\
		\hline
		\end{tabular}
		  \bigskip
	\caption{Values of $d(n,k,2)$ for small values of $n$ and $k$. Green values correspond to exceptional cases where $k\in A_{n,i}$ and $2k-2i-1\leq n \leq 2k+2i+1$. Yellow values correspond to the cases $2k-2i-1\leq n \leq 2k+2i+1$ with $k\notin A_{n,i}$. Red values correspond to the cases $n\geq 2k+2i+2$ or $n\leq 2k-2i-2$.}\label{tab:i2values}
\end{table}

Note that, in contrast with the $k=1,n-1$ case, Theorem~\ref{thm:skeleton} implies that there are three distinct behaviors for~$d(n,k,i)$ depending on the relative values of $n$, $k$ and $i$.
The first possibility is that there is no dimension-reducing projection of the hypersimplex that preserves all $i$-faces. This happens for some exceptional values captured by the restrictions $k\in A_{n,i}$ and $2k-2i-1\leq n\leq 2k+2i+1.$ These are the green values in the table. If $n$ is very large, or very small compared to $k$ and~$i$, then we can project to a space of fixed dimension, and achieve arbitrarily large codimension. These cases are shown in red in the table. Finally, there is an extra case (depicted in yellow) in which we can get a codimension~$1$ projection but no codimension~$2$ projection exists. It starts happening when $k\geq i+1$ and $n$ is of moderate size.

Actually, Theorem~\ref{thm:skeleton} is a corollary of our main result, Theorem~\ref{thm:fullcharacterization}, which provides the full characterization of the projections that attain these bounds. That is, we fully characterize which $n$-point configurations~$S$ verify that the $k$-set polytope $P_k(S)$ shares the $i$-skeleton with~$\Hnk$. Our characterization shows a surprising dichotomy: either $S$ is neighborly enough, or it has few vertices and it is not too %
neighborly, see Section~\ref{sec:hypersimplicial-neighborly} for details. 

In Section \ref{sec:strongce} we exploit this characterization to solve a variant of the convex embedding problem, also posed in the work of Halman, Onn and Rothblum~\cite{OnnRothblum2004}, for which we require the images of the vertices of the hypergraph to be in convex position as well.

Finally, in Section~\ref{sec:extremal} we study the associated extremal problem of maximizing the number of $k$-barycenters in convex position, which has been largely studied for graphs~\cite{Bilka2010,EisenbrandPachRothvossSopher2008,GarciaMarcoKnauer2016,HalmanOnnRothblum2007,SwanepoelValtr2010}. Let \defn{$g_k(n,d)$} be the maximum number of hyperedges that a $k$-uniform hypergraph on $n$ vertices with a convex embedding into~$\RR^d$ can have.

The function $g_k(n,d)$ exhibits three different regimes according to the value of~$d$. Our knowledge of the (asymptotic) growth of $g_k(n,d)$ for fixed $d$ and $k$ has different levels of precision for the three cases:
\bigskip

\begin{theorem}\label{thm:extremal}
For fixed values of $k$ and $d$, the value of $g_k(n,d)$ behaves as follows:\begin{itemize}
 \item If $d\geq 2k$, then $$g_k(n,d)=\binom{n}{k}.$$
 \item If $k+1\leq d\leq 2k-1$, then $g_k(n,d)$ is in $\Theta(n^k)$. More precisely, $$g_k(n,d)=\gamma_{k,d}\cdot n^k+o(n^k)$$ for some constant $\gamma_{k,d}$ satisfying 
 
 \[\binom{d}{k}\frac{1}{d^k}\leq \gamma_{k,d} \leq \frac{1}{k!}\left(1 - \frac{1}{\binom{n_{d,k}}{k-1}}\right)\]
 where 
 \[	n_{k,d}=\begin{cases}
		d+2 &\text{ if $d\geq 2k-3$, }\\
		\floor{\frac{d}{2}}+k &\text{ if $1\leq d\leq 2k-4$}.
	\end{cases}
	\]
 for $d\neq 1$, 
 and $n_{1,1}=2$ and $n_{k,1}=k$ for $k\geq 2$.
 \item If $d\leq k$, then $g_k(n,d)$ is
 in $\Omega(n^{d-1})$.\footnote{We originally claimed an $O(n^d)$ upper bound for this case in~\cite{jctb2021}, but our proof contained a mistake, see Remark~\ref{rmk:upperboundk>=k}.}
\end{itemize}

The limits in the $o$, 
$\Omega$ and $\Theta$ notations are taken as $n\to \infty$.
\end{theorem}

To put this result into perspective, we compare it with the known results for the case of graphs ($k=2$):
\begin{itemize}
 \item For $d\geq 4$, every graph has a convex embedding into $\RR^d$, and $g_2(n,d)=\binom{n}{2}$.
  \item For $d=3$, it is shown in~\cite{SwanepoelValtr2010} that $\gamma_{2,3}\in\{\frac{1}{3},\frac{3}{8}\}$, and $\gamma_{2,3}=\frac{1}{3}$ is conjectured, which would be the case if and only if for some $m$ there is no convex embedding of $K_{m,m,m,m}$ into $\RR^3$. Recently, Raggi and Rold\'an-Pensado found a convex embedding of $K_{2,2,2,2}$ into $\RR^3$ using computational methods (personal communication). Note that for these parameters, our results recover exactly these same lower and upper bounds: $\frac{1}{3}\leq \gamma_{2,3}\leq \frac{3}{8}$.
 \item For $d=2$, Halman, Onn and Rothblum asked whether $g_2(n,2)$ was linear or quadratic~\cite{HalmanOnnRothblum2007}. The answer is that $g_2(n,2)\in \Theta(n^{4/3})$, obtained as a result of the combined effort of diverse research teams. The tight upper bound was obtained in~\cite{EisenbrandPachRothvossSopher2008}, using
a generalization of the Szemer\'edi-Trotter Theorem for points and ``well-behaved'' curves in the plane~\cite{PachSharir1998}, and a matching lower bound was given in~\cite{Bilka2010} using configurations with the extremal number of point-line incidences.
\end{itemize}

In general, the \defn{$k$-set problem} asks for the maximal possible number of vertices of~$P_k(S)$ in fixed dimension $d$.
As Sharir, Smorodinsky, and Tardos put it, this is ``one of the most intriguing open problems in combinatorial geometry''~\cite{SST2001}, and there is a considerable gap between the upper and lower bounds (see~\cite{Wagner2008} for an extensive survey on the subject, and \cite{Sharir2011} for the latest improvement). 

Despite their close relation, the $k$-set problem and the estimation of $g_k(n,d)$ are fundamentally
distinct problems.
Indeed, note that for a general $k$-uniform hypergraph $H=(V,E)$, convex embeddings are more permissive than asking for the the $k$-barycenters induced by $E$ to be vertices of the $k$-set polytope. Indeed, the subsets of $k$-barycenters given by~$E$ can be in convex position even if they are not vertices of the whole set of $k$-barycenters. 

The $k$-set problem concerns the estimation of 
\defn{$a_k(n,d)$}, 
the maximal possible number of vertices of $P_k(S)$ for an $n$-point-set~$S$ in $\RR^d$, whereas we are studying the largest possible subset of $k$-barycenters that are in convex position. 
We trivially have \[a_k(n,d)\leq g_k(n,d),\]
but the converse is far from being true. For example, $a_2(n,2)=O(n)$ (see~\cite[Ch.~11]{MatousekLODG}), whereas $g_2(n,2)\in\Theta(n^{4/3})$~\cite{Bilka2010,EisenbrandPachRothvossSopher2008}.

In Section \ref{sec:open} we provide an additional discussion of our results and we collect various open problems for further work.

\section{Projections that strictly preserve the vertices of the hypersimplex}
\label{sec:projections}

In this section we reformulate Theorem~\ref{thm:main} in terms of polytope projections that preserve vertices. By projections we mean affine maps, although it suffices to focus on surjective linear maps. We assume some familiarity with the basic notions on polytope theory and refer the reader to~\cite{ZieglerBOOK} for a detailed treatment of the subject. Our main concern is the study of faces strictly preserved%
\footnote{We are mainly interested in strictly preserved faces, and our definition coincides with that in~\cite{RoerigSanyal2012,Sanyal2009,Ziegler2004}. However, our definition of (not necessarily strictly) preserved face differs from that in~\cite{RoerigSanyal2012,Sanyal2009}, where they require conditions \ref{it:piFface} and~\ref{it:samedim} to define preserved faces. We prefer this definition because it provides a bijection between faces of $\pi(P)$ and preserved faces, and simplifies the notation for Section~\ref{sec:ksets}.
}
under a projection, a notion introduced in Definition~3.1 of~\cite{Ziegler2004}. Note that if $\pi$ is an affine map and $P$ is a convex polytope, then the preimage of every face of $\pi(P)$ is a face of~$P$. These are called the preserved faces of $P$. And if a preserved face is affinely isomorphic to its image, then it is called strictly preserved.

\begin{definition}%
Let~$P$ be a polytope and $\pi:P\to\pi(P)$ a projection. A face $F\subseteq P$ is
\defn{preserved} under~$\pi$ if 
\begin{enumerate}[(i)]
 \item\label{it:piFface} $\pi(F)$ is a face of $\pi(P)$, and 
 \item\label{it:pi-1F} $\pi^{-1}(\pi(F))=F$. 
\end{enumerate}
If moreover 
\begin{enumerate}[(i),resume]
 \item\label{it:samedim} the map $F\to \pi(F)$ is a bijection,
\end{enumerate}
then we say that $F$ is \defn{strictly preserved} under~$\pi$. (This condition can be replaced by the equivalent conditions (iii') $\dim F=\dim \pi(F)$, or (iii'') $\pi(F)$ is combinatorially isomorphic to~$F$.) 
\end{definition}

For the restatement of Theorem~\ref{thm:main} we use the following auxiliary lemma that relates convex embeddings to projections of hypergraph polytopes. Let $H=(V,E)$ be a hypergraph. In analogy to edge polytopes~\cite{OhsugiHibi1998,Villarreal1998}, we define the \defn{hyperedge polytope $\hep{H}$} as the convex hull of the incidence vectors of the hyperedges of $H$:
\[\hep{H}:=\conv\Bset{\sum_{v\in e}\bfe_v}{e\in E}\subset\RR^{V},\]
where $\bfe_v\in \RR^{V}$ represent the standard basis vectors indexed by the vertices of~$H$. Note that $\hep{H}$ is a $0/1$-polytope.

\begin{lemma}
	\label{lem:equivalence}
	The existence of a convex embedding of a $k$-uniform hypergraph $H=(V,E)$ into $\RR^d$ is equivalent to the existence of an affine map of the hyperedge polytope $\hep{H}$ to $\RR^{d}$ that strictly preserves all its vertices. 
\end{lemma}

\begin{proof}
Let $V=\{v_1,\ldots,v_{n}\}$ be the vertex set of $H$. To any map $f:V\to \RR^d$ we associate the linear map $\pi:\RR^{V}\to \RR^{d}$ given by $\pi(\bfe_v)=\frac{1}{k}f(v)$. Notice that $\pi$ maps the vertices of $\hep{H}$ to the barycenters of the corresponding $k$-subsets of~$f(V)$. 
Now, these are in convex position if and only if $\pi(\hep{H})$ has $|E|$ distinct vertices (note that the convex position forces them to be distinct). Since $\hep{H}$ has $|E|$ vertices, this implies that the preimage of each vertex of $\pi(\hep{H})$ must be a single vertex of $\hep{H}$, and hence that all the vertices of $\hep{H}$ are strictly preserved by~$\pi$. The reverse argument associates affine maps strictly preserving the vertices of $\hep{H}$ to maps $f:V\to \RR^d$ with the $k$-barycenters of~$H$ in convex position. Since this is an open condition, if needed we can perturb~$f$ to force its injectivity and recover a convex embedding of~$H$.
\end{proof}

The hyperedge polytope of the complete $k$-uniform hypergraph is known as the \defn{$(n,k)$-hypersimplex} and denoted~\defn{$\Hnk$}. Namely, it is the polytope:
\[\Hnk:=\hep{\Knk}=\conv\Bset{x\in\{0,1\}^n}{\sum_{1\leq j\leq n}x_j=k}.\]

\begin{cor}
	\label{cor:equivalence}
	The existence of a convex embedding of $\Knk$ into $\RR^d$ is equivalent to the existence of an affine map of the hypersimplex $\Hnk$ to $\RR^{d}$ that strictly preserves its $\binom{n}{k}$ vertices. 
\end{cor}

Said differently, the projection $\pi:\RR^n\to\RR^d$ strictly preserves the vertices of~$\Hnk$ if and only if the point configuration \defn{$\Ppi:=\set{\pi(\bfe_i)}{1\leq i\leq n}$} has all its $k$-barycenters in convex position.

Corollary~\ref{cor:equivalence} implies that Theorem~\ref{thm:main} is a corollary of Theorem~\ref{thm:skeleton} obtained by setting $i=0$. Thus, from now on we focus on proving Theorem~\ref{thm:skeleton}. 

For a $d$-polytope $P\subset\RR^d$ and a linear surjection $\pi:\RR^d\to\RR^e$, 
the Projection Lemma \cite[Prop.~3.2]{Ziegler2004} %
gives a criterion to characterize which faces of~$P$ are strictly preserved by~$\pi$
in terms of an \defn{associated projection} $\tau:\RR^d\to\RR^{d-e}$. More precisely, let $\iota:\ker(\pi)\cong \RR^{d-e}\hookrightarrow \RR^d$ be the inclusion map of $\ker(\pi)$. Then $\tau$ is the adjoint map $\iota ^\ast: (\RR^d)^\ast \to (\RR^{d-e})^\ast$ after the canonical identifications $(\RR^d)^\ast \cong \RR^d$ and $(\RR^{d-e})^\ast \cong \RR^{d-e}$ (see~\cite[Sec.~3.2]{Sanyal2009} for details).

\begin{lemma}[{{Projection Lemma~\cite[Prop.~3.2]{Ziegler2004}}}]\label{lem:projection_lemma}
	Let $P\subset\RR^d$ be a $d$-polytope, $\pi:\RR^d\to\RR^e$ a linear surjection, and $\tau:\RR^d\to\RR^{d-e}$ be the associated projection.
	
	Let $F\subset P$ be a face of $P$ and let $I$ be the set of  normal vectors to the facets of $P$ that contain $F$. Then $F$ is strictly preserved if and only
	if $\set{\tau(\bfn)}{\bfn\in I}$ positively span $\RR^{d-e}$; i.e.\ if $0\in \intr\conv\set{\tau(\bfn)}{\bfn\in I}$.
\end{lemma}

As final ingredients we need the dimension and hyperplane description of $\Hnk$, as well as its facial structure. These are well known (see for example~\cite[Ex.~0.11]{ZieglerBOOK}).

\begin{lemma}\label{lem:hypersimplex}
The hypersimplex $\Hnk$ is
the polytope  
\[
	\Hnk=\Big\{\sum_{j\in[n]}x_j=k\Big\}\cap \bigcap_{j\in[n]} \Big\{x_j\geq 0\Big\}\cap\bigcap_{j\in[n]} \Big\{x_j\leq 1\Big\}.
	\]
It has $\binom{n}{k}$ vertices, which are the points in $\{0,1\}^n$ whose coordinate sum is~$k$. 
It is a point for $k\in\{0,n\}$, an $(n-1)$-simplex for $k\in\{1,n-1\}$, and for $2\leq k\leq n-2$ it is $(n-1)$-dimensional and has $2n$ facets.

For $1\leq i\leq n-1$, its $i$-faces are of the form 
\[
	\Hnk^{I,J}:=\Hnk\cap\bigcap_{i\in I} \Big\{x_i = 1\Big\} \cap \bigcap_{j\in J} \Big\{x_j= 0\Big\},
	\]
where $I,J\subset[n]$ are disjoint index sets with $|I|\leq k-1$, $|J|\leq n-k-1$ and $|I|+|J|=n-i-1$. The $i$-face $\Hnk^{I,J}$ is isomorphic to $\Delta_{n-|I|-|J|, k-|I|}$.
\end{lemma}

From here, we proceed as follows. Consider $n$ fixed and $2\leq k\leq n-2$. To work with a full $(n-1)$-dimensional polytope we identify $\Hnk$ with its projection onto $\RR^{n} / (\RR \cdot \uno) \cong \RR^{n-1}$, where $\uno$ represents the all ones vector. We want to study when there is an \defn{$i$-preserving projection} $\pi:\RR^{n-1}\to\RR^{d}$, that is, one that strictly preserves every face of the $i$-skeleton of~$\Hnk$.  If so, Lemma~\ref{lem:projection_lemma} would ensure certain positive dependencies on the vector configuration induced by the image of the normal vectors to facets of $\Hnk$ under the associated projection~$\tau$. We state explicitly these dependencies below. In Section~\ref{sec:hypersimplicial-neighborly}
we provide an in-depth study of the point configurations that yield vector configurations satisfying these dependencies and show that if $d$ is not large enough, then they cannot all hold simultaneously.

By the description in Lemma \ref{lem:hypersimplex}, for $2\leq k\leq n-2$, $\Hnk \subset \RR^{n} / (\RR \cdot \uno)$ has $2n$ facets whose normal vectors we may pair up as $\{\bfm_j,\bfn_j\}$ for $j\in[n]$, where $\bfm_j$ and $\bfn_j$ correspond to the inequalities $x_j\geq 0$ and $x_j\leq 1$ after the projection onto $\RR^{n} / (\RR \cdot \uno)$, respectively. They satisfy 
\begin{equation} \label{eq:symsum}
		\bfm_j+\bfn_j=0 \text{ for $j\in [n]$ and }
	\sum_{j\in[n]}\bfm_j =\sum_{j\in[n]}\bfn_j =0.
\end{equation}

\begin{example}
 Before we continue, we provide a concrete example of our set-up. Consider Figure \ref{fig:ziegler}. At the top of the figure we have the hypersimplex $\Delta_{4,2}$, which is a $3$-dimensional octahedron. By construction, its ambient space is $\mathbb{R}^4$, but we isomorphically project it onto $\RR^{4} / (\RR \cdot (1,1,1,1))\cong \RR^{3}$. Thus, when projecting it to the plane we get a map $\pi:\mathbb{R}^3\to \mathbb{R}^2$.

\begin{figure}[htpb]
	\includegraphics[width=1\linewidth]{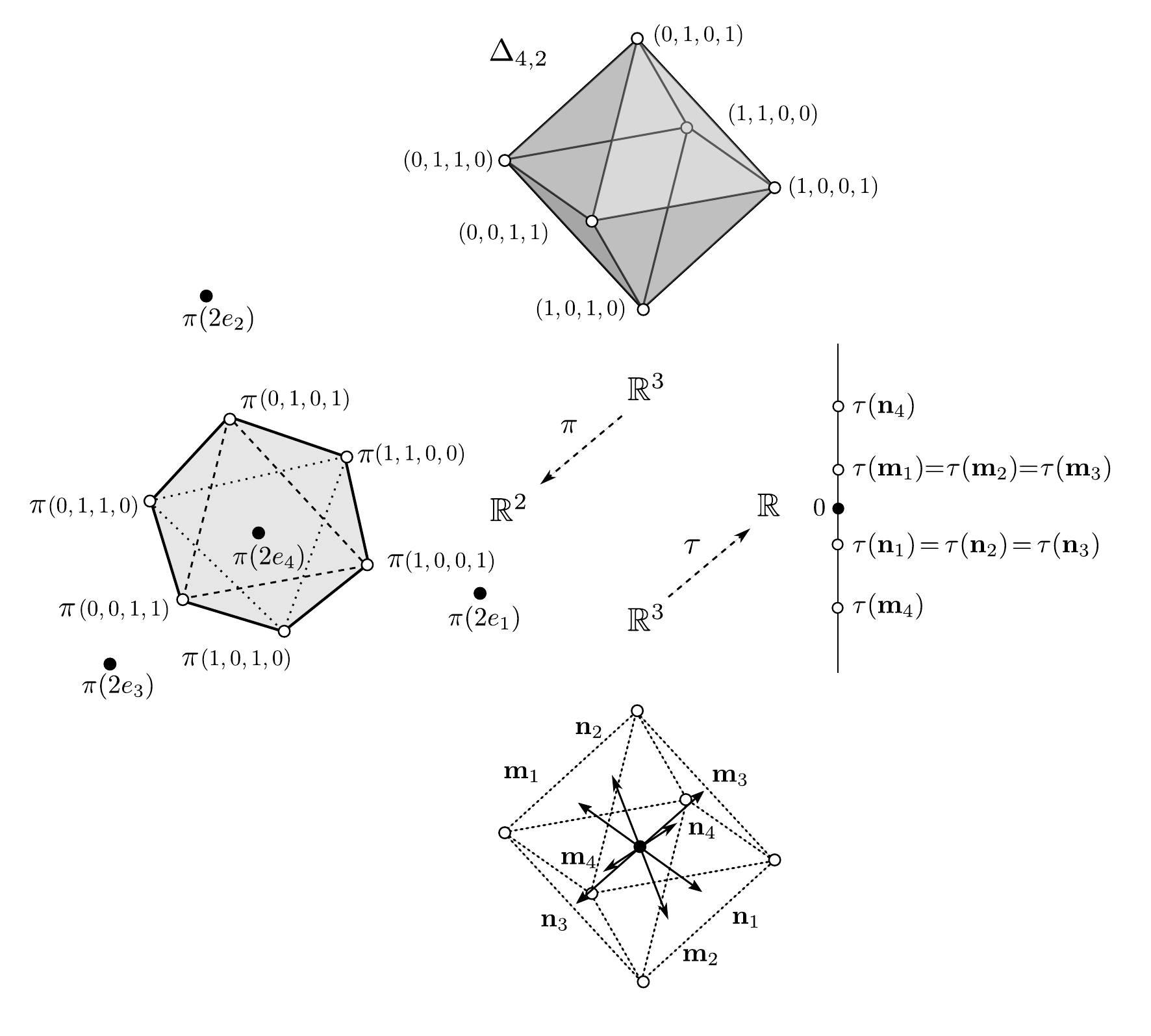}
	\caption{Using the Projection Lemma on a projection for the hypersimplex $\Delta_{4,2}$, which is an octahedron.}\label{fig:ziegler}
\end{figure}

At the left side of the figure we have the image of $\Delta_{4,2}$ under $\pi$. We also show the images of $2\bfe_1, 2\bfe_2, 2\bfe_3, 2\bfe_4$. Note that, as expected by Corollary~\ref{cor:equivalence}, the images of the vertices of $\Delta_{4,2}$ are precisely the midpoints of the edges $\pi(2\bfe_i)\pi(2\bfe_j)$, for $1\leq i<j\leq 4$, all of which lie in strictly convex position. Even though this is evident from the figure, we can also verify it using the Projection Lemma.

To do so, consider the normal vectors to the faces of $\Delta_{4,2}$. These are shown at the bottom of the figure with labels $\bfm_j$, $\bfn_j$ for $j=1,2,3,4$. Since $\pi$ has codimension~$1$, it induces a projection $\tau$ that takes these normal vectors to~$\RR$, which is shown at the right side of the figure. So, consider for example the vertex $(0,1,0,1)$. This vertex lies on faces with normal vectors $\bfm_1,\bfn_2,\bfm_3,\bfn_4$, and $0$ is strictly contained in the interior of the convex hull of $\tau(\{\bfm_1,\bfn_2,\bfm_3,\bfn_4\})$. Thus, by Lemma \ref{lem:projection_lemma} we verify that $(0,1,0,1)$ is strictly preserved under $\pi$.

The Projection Lemma also determines whether higher dimensional faces are preserved or not. Consider for example the edge $(0,1,0,1)(1,0,0,1)$ of $\Delta_{4,2}$. It is contained in the faces with normal vectors $\bfm_3$ and $\bfn_4$. Note that $0$ does not lie in the interior of the convex hull of $\tau(\{\bfm_3,\bfn_4\})$. By Lemma~\ref{lem:projection_lemma} we conclude that $\pi$ does not preserve the edge, which can be verified by inspection. We invite the reader to check the preservation of the remaining vertices and faces.
\end{example}

To use these tools in general, we combine Lemma~\ref{lem:projection_lemma} with the facial structure of the hypersimplex to get:

\begin{lemma}
	\label{lem:configuration}

	Let $\pi:\RR^{n-1}\to\RR^{d}$ be a linear surjection and~$\tau:\RR^{n-1}\to \RR^{n-d-1}$
	its associated projection.
	Let $2\leq k\leq n-2$, and let $\set {\bfn_j,\bfm_{j'}}{j,j'\in[n]}$ be the normal vectors of~$\Hnk$ described above.
	Then 
	$\tau(\set {\bfn_j,\bfm_{j'}}{j,j'\in[n]})$ 
	is an $(n-d-1)$-dimensional configuration of vectors with the following strictly positive dependencies:
	\begin{enumerate}[a)]
		\item $0\in \intr\conv\{\tau(\bfm_j): j\in[n]\}$,
		\item $0\in \intr\conv\{\tau(\bfn_j): j\in[n]\}$,
		\item $0\in \relint\conv\{\tau(\bfm_j),\tau(\bfn_j)\}$ for $j\in[n]$.
	\end{enumerate}

	The projection~$\pi$ is $i$-preserving if and only if:
	\begin{itemize}[leftmargin=\parindent]
	 \item 	When $i=0$, additionally to a), b) and c), for every disjoint $I,J\subset[n]$ with $|I|=k$ and $|J|=n-k$, we have

	\begin{enumerate}
		\item[d)] $0\in \intr\conv\left(\{\tau(\bfm_j): j\in J\}\cup \{\tau(\bfn_j): j\in I\}\right)$ and 
		\item[e)] $0\in \intr\conv\left(\{\tau(\bfm_j): j\in I\}\cup \{\tau(\bfn_j): j\in J\}\right)$.
	\end{enumerate}

        \item When $i\geq 1$, additionally to a), b) and c), for every disjoint $I,J\subset[n]$ such that $|I|\leq k-1$, $|J|\leq n-k-1$ and $|I|+|J|=n-i-1$, we have

	\begin{enumerate}
		\item[d')] $0\in \intr\conv\left(\{\tau(\bfm_j): j\in I\}\cup \{\tau(\bfn_j): j\in J\}\right)$ and
		\item[e')] $0\in \intr\conv\left(\{\tau(\bfm_j): j\in J\}\cup \{\tau(\bfn_j): j\in I\}\right)$.
	\end{enumerate}

	\end{itemize}

\end{lemma}

\begin{proof}
	The positive dependencies in $a)$, $b)$ and $c)$ follow directly from the linearity of $\tau$ and~\eqref{eq:symsum}. Note that \eqref{eq:symsum} additionally states that the vector configuration is symmetric around the origin with pairing $\tau(\bfm_j)=-\tau(\bfn_j)$.

	For $d)$ we use that $\pi$ preserves the vertices of $\Hnk$. Each vertex of $\Hnk$ lies in exactly $k$ hyperplanes of the form $x_j=1$ and $n-k$ hyperplanes of the form $x_j=0$. From here we obtain, respectively, complementary index sets $I$ and $J$ of $[n]$. The conclusion then follows from Lemma \ref{lem:projection_lemma}.

	The analysis for $d')$ is similar considering the description of the $i$-faces of~$\Hnk$ given in Lemma~\ref{lem:hypersimplex}.

	Finally, the family of positive dependencies in $e)$ and $e')$ follow respectively from $d)$ and $d')$ and the symmetry around the origin.
\end{proof}
 
Since the configuration of vectors is symmetric around the origin, we obtain a proof of the following observation.

\begin{cor}
\label{cor:symmetry}
	An $i$-preserving projection $\pi:\RR^{n-1}\to\RR^{d}$ exists for $\Hnk$ if and only if it exists for $\Delta_{n,n-k}$.
\end{cor}

Of course, $\Hnk$ and $\Delta_{n,n-k}$ are affinely equivalent, so Corollary~\ref{cor:symmetry} should not be too unexpected. However, the fact that $\cd(\Knk)=\cd(K_n^{(n-k)})$ is not entirely obvious from the definition of $\cd$. It has an alternative short geometric proof. Suppose $f$ is a convex embedding of $K_{n}^{(k)}$ into $\RR^d$. Consider the barycenter $b$ of $f(V)$. The barycenter $a$ of any $k$-subset of $f(V)$, the barycenter $c$ of the complementary $(n-k)$-subset and $b$ are collinear. The segment $ac$ is split in ratio $k:n-k$ by~$b$. Therefore, the set of $(n-k)$-barycenters is a homothetic copy of the set of $k$-barycenters. Since the second is in convex position, the first one is as well.

\section{Hypersimplicial-neighborly configurations}\label{sec:hypersimplicial-neighborly}

In this section we prove Theorem~\ref{thm:skeleton}. By Corollary \ref{cor:symmetry} we may focus only on the cases $n\geq 2k$, and therefore it is enough to prove the following:
\[
	d(n,k,i)=\begin{cases}
		2k+2i &\text{for $n\geq 2k+2i+2$}\\
                n-1 &\text{if $2k\leq n \leq 2k+2i+1$, $k\in A_{n,i}$},\\
                n-2 &\text{if $2k\leq n \leq 2k+2i+1$, $k\notin A_{n,i}$}.\\
	\end{cases}
\]
	Recall 
	that $A_{n,i} = \{1,2,\ldots,i+1\}\cup \{n-i-1,n-i,\ldots,n-1\}$ is the range of some exceptional values for $k$.

We actually prove a stronger result, as we completely characterize all projections $\pi:\RR^{n-1}\to \RR^d$ that preserve the $i$-skeleton of~$\Delta_{n,k}$. 
In the remainder of the section we assume $n\geq 2k\geq 2$ and, as before, $\pi$ is linear and surjective and $\tau:\RR^{n-1}\to \RR^{n-d-1}$ denotes the projection associated to~$\pi$.

Our characterization is in terms of the point configuration \defn{$\Ppi:=\set{\pi(\bfe_i)}{i\in [n]}$}%
; more precisely, in terms of its neighborliness and almost neighborliness. A point configuration~$S$ is called \defn{$j$-neighborly} if every subset of at most~$j$ points of~$S$ is the vertex set of a face of $\conv(S)$ (thus $1$-neighborly corresponds to being in convex position), and~$S$ is \defn{$j$-almost neighborly} if every subset of at most $j$ points of $S$ lies in a common face of $\conv(S)$.

The relation with neighborliness was already observed by Halman, Onn and Rothblum, who used cyclic $2k$-polytopes (which are $k$-neighborly) to provide convex embeddings of $k$-uniform hypergraphs into~$\RR^{2k}$~\cite{HalmanOnnRothblum2007}. Their observation can be extended to higher dimensional skeleta of the hypersimplex.
\begin{lemma}\label{lem:neighborlyworks}
If $1\leq k\leq \frac n2$ and $\Ppi$ is $(k+i)$-neighborly, then $\pi$ is an $i$-preserving projection of~$\Hnk$.
\end{lemma}

\begin{proof}
	Since $\Ppi$ is $(k+i)$-neighborly, any set of at most $k+i$ vertices forms a simplex face of~$\Ppi$.

	For a subset $A\subset [n]$ of size $k+i$, consider the faces $F=\Hnk\cap_{j\notin A} \{x_j=0\}$ and $G=\conv\set{\pi({\bfe_i})}{i\in A}$ of $\Hnk$ and~$\Ppi$, respectively. Then $\pi$ restricted to the affine span of~$F$ is an affine isomorphism into the affine span of~$G$. The supporting hyperplane for $G$ in~$\Ppi$ is also supporting for $\pi(F)$ in $\pi(\Hnk)$, and hence $F$ is strictly preserved.
	
	We conclude the proof by observing that any $i$-face of $\Hnk$ belongs to one such face~$F$. 
\end{proof}

We will also use Gale duality (see \cite[Lec.~6]{ZieglerBOOK} or \cite[Sec.~5.6]{MatousekLODG} for nice introductions). 
A short computation leads to the following observation. It can also be easily seen by comparing our coordinate-free definition for~$\tau$ with Ewald's introduction to Gale transforms~\cite[Sec~II.4]{Ewald}. We omit the details.

\begin{lemma}\label{lem:Gale}
 The vector configuration $M=\set{\tau(\bfm_i)}{1\leq i\leq n}$ is a Gale transform of $\Ppi$.
\end{lemma}

A first consequence of this observation is that, for a point configuration $S$, the property of having all $k$-barycenters in convex position only depends on its underlying oriented matroid (in particular, it is invariant under admissible projective transformations, i.e.\ those where the hyperplane at infinity does not separate~$S$). Indeed, by the Projection Lemma~\ref{lem:projection_lemma} (and Lemma~\ref{lem:configuration}), the facial structure of $P_k(S)$, the convex hull of the $k$-barycenters, only depends on the oriented matroid of the vector configuration $\set{\tau(\bfm_j),\tau(\bfn_j)}{j\in[n]}$, which is completely determined by the oriented matroid of $M$ by the central symmetry. This might be counter-intuitive at first, as barycenters are not preserved by projective transformations. However, the interpretation in terms of $k$-sets in Section~\ref{sec:ksets} gives a clear explanation for why the combinatorial type of~$P_k(S)$ is an oriented matroid invariant.

We will use the Gale dual characterization of neighborliness and almost neighborliness. This result is well known (see for example~\cite{NillPadrol2015}), but we include a short schema for the proof for completeness.

\begin{lemma}\label{lem:Galeneigh}
 $\Ppi$ is $j$-neighborly if and only if every open linear halfspace contains at least $j+1$ vectors of $M$, and $\Ppi$ is $j$-almost neighborly if every closed linear halfspace contains at least $j+1$ vectors of $M$.
\end{lemma}
\begin{proof}
 Using standard properties of the Gale transform characterizing faces~\cite[Cor.~5.6.3]{MatousekLODG}, $j$-neighborliness of $\Ppi$ means that every subset of $j$ points is the vertex set of a simplex face of $\Ppi$, which is equivalent to $0$ being in the  interior of the convex hull of every subset of $n-j$ points of~$M$. By Farkas' lemma, this is equivalent to every linear open halfspace containing at least one point of each subset of~$M$ of cardinality $n-j$, which is equivalent to the stated condition. For almost-neighborliness, we remove the condition of being in the interior, which translates to closed halfspaces instead.
\end{proof}

This provides an easy alternative proof of Lemma~\ref{lem:neighborlyworks}. We present it for convenience to the reader in order to provide extra insight into our later arguments.
\begin{proof}[Dual proof for Lemma~\ref{lem:neighborlyworks}]
If $i\geq 1$, let $I,J\subset [n]$ be disjoint subsets such that $|I| +|J|=n-i-1$, $|I|\leq k-1$, and $|J|\leq n-k-1$. And if $i=0$, let $I,J\subset [n]$ be disjoint subsets such $|I|= k$, and $|J|= n-k$. In both cases we have $|J|\geq n-i-k$. To prove our claim, it suffices to show that $0\in \intr \conv(\set{\tau(\bfm_j)}{j\in J})$. By Farkas' Lemma, this is equivalent to showing that every linear open halfspace $H^+$ contains at least one $\tau(\bfm_j)$ with $j\in J$. This holds because $|H^+\cap M|\geq k+i+1$ (by Lemma~\ref{lem:Galeneigh}) and $J$ is missing at most $i+k$ vectors from $M$.
\end{proof}

This viewpoint also allows us to provide another family of examples of $i$-preserving projections.

\begin{lemma}\label{lem:almostneighborlyworks}
If $1\leq k\leq \frac n2$ and $\Ppi$ is an $(n-2)$-dimensional configuration of $n$ points that is not $(k-i-1)$-almost neighborly, then 
$\pi: \RR^{n-1}\to \RR^{n-2}$ is an $i$-preserving projection of codimension~$1$ of~$\Hnk$.
\end{lemma}

\begin{proof}
Let $M^+:=\set{j\in[n]}{\tau(\bfn_j)> 0}$, $M^-:=\set{j\in[n]}{\tau(\bfn_j)< 0}$, and $M^0:=\set{j\in[n]}{\tau(\bfn_j)=0}$.
 
By Lemmas~\ref{lem:Gale} and~\ref{lem:Galeneigh}, $\Ppi$ not being $(k-i-1)$-almost neighborly is equivalent to 
	$\min\{|M^+\cup M^0|,|M^-\cup M^0|\}\leq k-i-1.$

We may assume that $|M^+\cup M^0|\leq |M^-\cup M^0|$, so $|M^+\cup M^0|\leq k-i-1$. We pick $I,J\subset[n]$ such that $I\cap J=\emptyset$, $|I|\leq k-1$, $|J|\leq n-k-1$ and $|I|+|J|=n-i-1$. Since $|J|\leq n-k-1$, we have that $|I|\geq k-i$. Since $n\geq 2k$, $|J|\geq n-k-i\geq k-i$. 

Therefore, we may choose $j\in I\cap M^-$ and $j'\in J\cap M^-$, which verify $\tau(\bfm_j)<0$ and $\tau(\bfm_{j'})<0$. So $\tau(\bfn_j)=-\tau(\bfm_j)>0$. Therefore, $0$ is in the desired interior of the convex hull.
\end{proof}

We can give an explicit description of the examples provided by Lemma~\ref{lem:almostneighborlyworks}. Full-dimensional configurations of $d+2$ points in~$\RR^d$ are well classified~\cite[Sec.~6,1]{Grunbaum}. They are all (up to admissible projective transformation) of the form $\pyr_k(\Delta_n\oplus\Delta_m)$, for $k,n,m\in\NN$ such that $k+n+m=d$. Here \defn{$\Delta_n$} represents (the vertex set of) an \defn{$n$-dimensional simplex}; the \defn{direct sum $P \oplus Q$} is the $\left(\dim(P)+\dim(Q)\right)$-dimensional configuration obtained by taking a copy of $P$ and $Q$ whose convex hulls intersect in a point in the relative interior of both; and the \defn{$k$-fold pyramid $\pyr_k(P)$} is the $\left(k+\dim(P)\right)$-dimensional configuration obtained by adding $k$ affinely independent points (so for $k=0$, we define $\pyr_0(P)=P$).
Note that $\pyr_k(\Delta_n\oplus\Delta_m)$ is $\min(n,m)$-neighborly (but not $(\min(n,m)+1)$-neighborly) and $(\min(n,m)+k)$-almost neighborly (but not $(\min(n,m)+k+1)$-almost neighborly). Hence $\Delta_{n-2}\oplus\Delta_0$, which consists of an $(n-2)$-simplex together with an interior point, corresponds to an $i$-preserving projection for $\Delta_{n,k}$ whenever $k\geq i+2$; whereas $\Delta_{\floor{\nicefrac n2}-1}\oplus\Delta_{\ceil{\nicefrac n2}-1}$ corresponds to an $i$-preserving projection if $k+i+1\leq \floor{\frac n2}$.

As it turns out, the configurations given by Lemmas~\ref{lem:neighborlyworks} and~\ref{lem:almostneighborlyworks} describe all possible $i$-preserving projections. We show this first
 for the cases of codimension~$1$ and then for those of larger codimension.

\begin{proposition}\label{prop:charcod1}
If $1\leq k\leq \frac n2$, the surjective projection $\pi:\RR^{n-1}\to \RR^{n-2}$ is $i$-preserving for $\Hnk$ if and only if either
\begin{enumerate}[(i)]
 \item\label{it:neigh} $\Ppi$ is $(k+i)$-neighborly, or
 \item\label{it:notalmostneigh} %
 $\Ppi$ is not $(k-i-1)$-almost neighborly.
\end{enumerate}

That is, if either
\begin{enumerate}[(i)]
 \item[\ref{it:neigh}] $\min\{|M^+|,|M^-|\}\geq k+i+1$, or
 \item[\ref{it:notalmostneigh}] %
 $\min\{|M^+\cup M^0|,|M^-\cup M^0|\}\leq k-i-1$;
\end{enumerate}
where 
\begin{align*}
	M^+&:=\set{j\in[n]}{\tau(\bfn_j)> 0},\\ 
	M^-&:=\set{j\in[n]}{\tau(\bfn_j)< 0} \text{and}\\
	M^0&:=\set{j\in[n]}{\tau(\bfn_j)=0}.
\end{align*}
\end{proposition}
\begin{proof}
We have already seen that if $\Ppi$ verifies \ref{it:neigh} or \ref{it:notalmostneigh}, then $\pi$ is $i$-preserving.
Now we show that no other projection $\pi$ of codimension~$1$ can be $i$-preserving. 
If $S_\pi$ is not $(k+i)$-neighborly, then $\min\{|M^+|,|M^-|\}\leq k+i$. Without loss of generality, say $|M^+|\geq |M^-|$ and $|M^-|\leq k+i$.
If moreover $S_\pi$ is $(k-i-1)$-almost neighborly, then $|M^+\cup M^0|\geq k-i$ and $|M^- \cup M^0|\geq k-i$. 
We also have
\[|M^+\cup M^0|=n-|M^-|\geq n-k-i.\]

Assume first $i\geq 1$. Now,
\begin{itemize}
        \item If $|M^+|\geq n-k-1$, let $J$ be a subset of $M^+$ of size $n-k-1$ and $I$ be a subset of $M^-\cup M^0$ of size $k-i$.
        \item If $|M^+| \leq n-k-2$, let $Z\subseteq M^0$ be a minimal (maybe empty) set such that $$n-k-i\leq |M^+\cup Z| \leq n-k-1,$$
        define $J=M^+\cup Z$ and let $I$ be a subset of $[n]\setminus J\subseteq M^-\cup M^0$ of size $n-i-1-|J|$.
\end{itemize}

If $i=0$, 
\begin{itemize}
        \item If $|M^+|\geq n-k$, let $J$ be a subset of $M^+$ of size $n-k$ and $I$ be a subset of $M^-\cup M^0$ of size~$k$.
        \item If $|M^+| \leq n-k-1$, let $Z\subset M^0$ be a set such that $$n-k = |M^+\cup Z|,$$
        define $J=M^+\cup Z$ and let $I=[n]\setminus J$.
\end{itemize}

In both cases, we obtain disjoint subsets $I$, $J$ that satisfy 
\begin{itemize}
		\item $|I|\leq k-1$, $|J|\leq n-k-1$, $|I|+|J|=n-i-1$ when $i\geq 1$,
		\item $|I|= k$, $|J|= n-k$ when $i=0$, and
        \item $\tau(\bfm_j)\geq 0$ for $j\in I$ and $\tau(\bfn_j)\geq 0$ for $j\in J$.
\end{itemize}

Then $0$ is not in the interior of the convex hull of $$\{\tau(\bfm_j): j\in I\}\cup \{\tau(\bfn_j): j\in J\},$$ and thus $\pi$ is not $i$-preserving by Lemma~\ref{lem:configuration}.
\end{proof}

\begin{proposition}\label{prop:charcod2}
If $1\leq k\leq \frac n2$, the projection $\pi:\RR^{n-1}\to \RR^{n-3}$ is $i$-preserving for $\Hnk$ if and only if $\Ppi$ is $(k+i)$-neighborly.

That is, if for every linear open halfplane $H^+$ we have 
\begin{equation}\label{eq:condneigh}|H^+\cap M|\geq k+i+1.\end{equation}
\end{proposition}

\begin{proof}
Let $M=\tau(\set{\bfm_j}{j\in[n]})\subset \RR^2$ and $N=\tau(\set{\bfn_j}{j\in[n]}) \subset \RR^2$. By construction, $M\cup N$ is a centrally symmetric $2$-dimensional vector configuration. And $\pi$ is $i$-preserving if and only if some prescribed positive dependencies are verified, by Lemma~\ref{lem:configuration}.
We will show that these dependencies are equivalent to condition~\eqref{eq:condneigh}.

The fact that~\eqref{eq:condneigh} implies that $\pi$ is $i$-preserving is a direct corollary of Lemmas~\ref{lem:neighborlyworks} and~\ref{lem:Galeneigh}. 
Hence, it suffices to prove the converse, that is, that if $\pi$ is $i$-preserving then~\eqref{eq:condneigh} holds. Note that it is sufficient to show that we have~\eqref{eq:condneigh} for the halfplanes supported by lines spanned by the vectors in~$M$.

Consider an oriented line $\ell$ through the origin. 
Let $A\subset [n]$ and $B\subset [n]$ be the indices of the vectors of $M$ strictly to the right and left of $\ell$ respectively. Let $C\subset [n]$ index the vectors of $M$ on the open ray from $0$ with the same direction as $\ell$ and $D\subset [n]$ the vectors of $M$ on the opposite ray (see Figure~\ref{fig:rotatingline}). Then $E=[n]\setminus (A\cup B\cup C\cup D)$ are the indices of the vectors that are copies of $0$. Let $a=|A|$, $b=|B|$, $c=|C|$, $d=|D|$, and $e=|E|$.

The same arguments of the proof of Proposition~\ref{prop:charcod1} show that we must have either 
\begin{enumerate}[(i)]
 \item[\ref{it:neigh}] $\min\{a,b\}\geq k+i+1$, or
 \item[\ref{it:notalmostneigh}] %
 $\min\{a+c+d+e,b+c+d+e\}\leq k-i-1$.
\end{enumerate}
Our claim is that it is always the first condition that holds. 

We continuously rotate $\ell$ clockwise. The sign of $a-b$ changes after half a complete rotation. This ensures that we can find a position of the line for which $a+c\geq b+d$, $b+c\geq a+d$ and $c+d\neq 0$, as the switch has to take place at one of the lines spanned by a vector of $M$.

\begin{figure}[htpb]
\includegraphics[width=.3\linewidth]{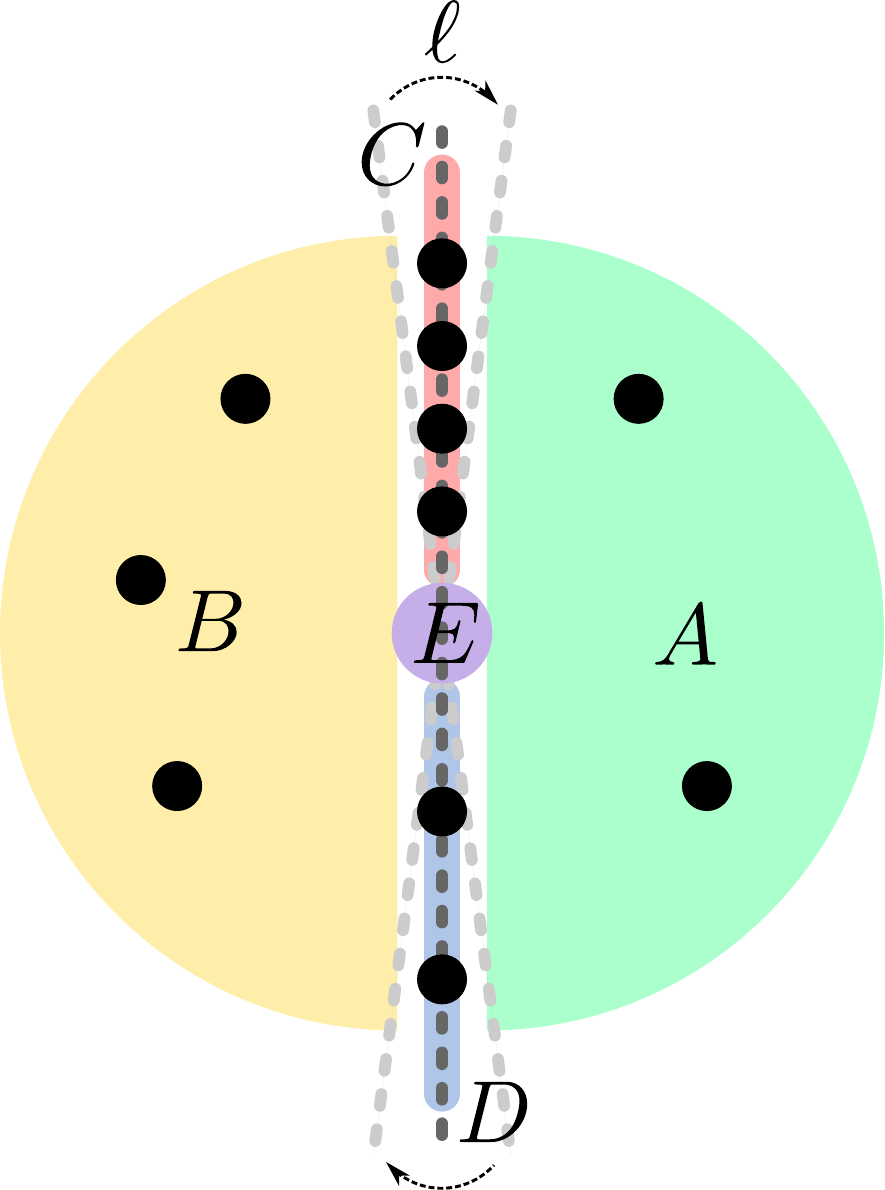}
\caption{The partition $A,B,C,D,E$ induced by an oriented line. In this example, $a=2$, $b=3$, $c=4$, $d=2$, and $e=0$, and we have $a+c=6\geq 5=b+d$, $b+c =7\geq 4=a+d$ and $c+d=6\neq 0$.}\label{fig:rotatingline}
\end{figure}

We claim that \ref{it:neigh} holds. Indeed, if $a+c+d+e\leq k-i-1$, then $b=n-(a+c+d+e)\geq 2k - (k-i-1)=k+i+1$. Hence, we have $b+d\geq k+i+1\geq k-i-1\geq a+c$, a contradiction. And analogously, if $b+c+d+e\leq k-i-1$, then $a+d\geq k+i+1\geq k-i-1\geq b+c$.

We conclude that there is at least one line for which \ref{it:neigh} holds. Now, assume that $\ell$ is a line for which \ref{it:neigh} holds, and let $\ell'$ be the next  line spanned by $M$ in clockwise order. It defines a new partition, with corresponding values $a',b',c',d',e'$. We have $a'=a-c'+d$, $b'=b-d'+c$ and $e'=e$. 
Thus,
\begin{align*}
a'+c'+d'+e'&=a+d+d'+e'\geq a\geq k+i+1> k-i-1, \text{ and }\\ b'+c'+d'+e'&=b+c+c'+e'\geq b\geq k+i+1> k-i-1.
\end{align*}
Which shows that \ref{it:notalmostneigh} cannot hold. And hence that \ref{it:neigh} must hold for $\ell'$, and, by induction, for all lines spanned by $M$. 
\end{proof}

\begin{cor}\label{cor:characterization}
If $1\leq k\leq \frac n2$ and $\ell\geq 2$, the projection $\pi:\RR^{n-1}\to \RR^{n-1-\ell}$ is $i$-preserving for $\Hnk$ if and only if $\Ppi$ is $(k+i)$-neighborly.

That is, if for every linear open halfspace $H^+$ we have 
\begin{equation}\tag{\ref{eq:condneigh}}|H^+\cap M|\geq k+i+1.\end{equation}

\end{cor}
\begin{proof}
Actually, the proof of Proposition~\ref{prop:charcod2} extends to the general case almost verbatim. The direct implication follows from Lemma~\ref{lem:neighborlyworks}.
For the converse, let $F$ be a $(d-2)$-dimensional flat spanned by vectors in $M\cap H$. Then we can repeat the argument of the proof of Proposition~\ref{prop:charcod2} by pivoting a hyperplane~$H$ containing~$F$ clockwise around $F$ to conclude that all hyperplanes containing~$F$ verify~\eqref{eq:condneigh}.
\end{proof}

\begin{theorem}\label{thm:fullcharacterization}
 If $1\leq k\leq \frac n2$ and $0\leq i\leq n-1$, the projection $\pi:\RR^{n-1}\to \RR^{d}$ is $i$-preserving for $\Hnk$ if and only if 
 \begin{enumerate}[(i)]
  \item $\Ppi$ is $(k+i)$-neighborly,
  \item\label{it:almostneigh} $\Ppi$ is $(n-2)$-dimensional and not $(k-i-1)$-almost neighborly, or
  \item $\pi$ is an affine isomorphism.
 \end{enumerate}

 In particular, such a projection exists only if
 \begin{enumerate}[(i)]
  \item $2k+2i\leq d\leq n-1$ (or $d=n-1\geq k+i$),
  \item $d=n-2$ and $k\geq i+2$, or
  \item $d=n-1$.
 \end{enumerate}
\end{theorem}
\begin{proof}
 If $\pi$ is an isomorphism, then $d=n-1$ and all the faces are preserved. In this case $\Ppi$ is just the vertex set of an $(n-1)$-simplex. 
 
 Every point configuration is at least $0$-almost neighborly, and hence the construction for~\ref{it:almostneigh} only makes sense if $k\geq i+2$. In this case, the configuration of $n$ points in $\RR^{n-2}$ consisting of the vertex set of an $(n-2)$-simplex and its barycenter provide an example of a configuration that is not $1$-almost neighborly. 

 Finally, it is well known that there are $j$-neighborly polytopes in all dimensions $d \geq 2j$. For example the cyclic polytopes which are the convex hulls of~$n$ points  in the \defn{moment curve} $\set{(t,t^2,\ldots,t^{d})}{t\in \RR}\subset\mathbb{R}^{d}$ are $\floor{\frac d2}$-neighborly~\cite[Cor.~0.8]{ZieglerBOOK}. It is also well known that no $d$-polytope other than the simplex is $\ell$-neighborly for $\ell>\floor{\frac{d}{2}}$~\cite[Exercice~0.10]{ZieglerBOOK}.
\end{proof}

Theorem~\ref{thm:skeleton} follows directly from this.

\begin{proof}[Proof of Theorem~\ref{thm:skeleton}]
If $n\geq 2k+2i+2$, then $2k+2i\leq n-2$ and an optimal embedding is in a $(k+i)$-neighborly polytope in~$\RR^{2k+2i}$. If $2k\leq n\leq 2k+2i+1$, then there is no neighborly embedding except for the isomorphism, and hence whether $d(n,k,i)=n-2$ or $d(n,k,i)=n-1$ depends on whether a non-almost neighborly embedding exists. This is determined by the condition $k\geq i+2$. That is if $k\notin A_{n,i}$, then the simplex with an interior point works, otherwise the only embedding is the isomorphism.
The remaining cases follow by symmetry.
\end{proof}

We want to end this section by thanking an anonymous reviewer who pointed us to an alternative proof of Theorem~\ref{thm:main} via a result of Shephard from~\cite{Shephard1969}. It states that every configuration of $d+3$ points in $\RR^d$ has a Radon partition into disjoint subsets of sizes $r$ and $s=d+3-r$ whose convex hulls intersect, for any $\floor{\frac{d+2}{2}}\leq r\leq\ceil{\frac{d+4}{2}}$. This shows that, when $n\geq 2k$, if $d<\min(2k,n-2)$, then there will be a $k$-element subset that cannot be a $k$-set. Our Proposition~\ref{prop:charcod2} can be seen as a generalization of Shephard's result, which follows from the $i=0$ special case. 

\section{Strong convex embeddings}
\label{sec:strongce}

Halman, Onn and Rothblum \cite{HalmanOnnRothblum2007} also define the following notions. A \defn{strong convex embedding} for a hypergraph $H=(V,E)$ is a convex embedding $f:V\to \RR^d$ for which $f(V)$ is also in convex position; and the \defn{strong convex dimension $\scd(H)$} of $H$ is the minimal $d$ for which a strong convex embedding of $H$ into $\RR^d$ exists.

Strong convex embeddings have some intrinsic interest (and the associated extremal problems have a different behavior, see Section~\ref{sec:open}). They also yield useful information on (normal) convex embeddings. For example, Swanepoel and Valtr~\cite{SwanepoelValtr2010} use the non-existence of a strong convex embedding for $K_5$ into $\RR^3$ as a key step in proving that there is no (normal) convex embedding for $K_{2,2,2,2,2}$ into~$\RR^3$.

Here we exploit our characterization of $i$-preserving projections for $\Delta_{n,k}$ to determine $\scd(\Knk)$. In fact, we determine the value of \defn{$d'(n,k,i)$}, the smallest dimension for which we can find a projection $\pi:\Hnk\to \RR^d$ that strictly preserves the $i$-dimensional skeleton of $\Hnk$ and for which the associated point configuration $S_\pi$ is in convex position.

\begin{cor}
	\label{cor:strong}
	Given positive integers $n$, $k$, $i$ such that $1\leq k\leq n-1$ and $0\leq i\leq n-1$, the value of $d'(n,k,i)$ is determined as follows.
		\[
		d'(n,k,i)=
			\begin{cases}
				2k+2i &\text{if $n\geq 2k+2i+2$},\\
				2n-2k+2i &\text{if $n\leq 2k-2i-2$},\\
                                n-1 &\text{if $2k-2i-1\leq n \leq 2k+2i+1$, $k\in C_{n,i}$},\\
				n-2 &\text{if $2k-2i-1\leq n \leq 2k+2i+1$, $k\notin C_{n,i}$}.\\
			\end{cases}
		\]
	Where
	\begin{align*}
		C_{n,i} &= \{1,2,\ldots,i+2\}\cup \{n-i-2, n-i-1,\ldots,n-1\}.\\
	\end{align*}
\end{cor}

Note that $C_{n,i}=A_{n,i}\cup \{i+2, n-i-2\}$, so we are saying that $d'(n,k,i)=d(n,k,i)$ except for the cases $2k-2i-1\leq n \leq 2k+2i+1$ and $k\in \{i+2, n-i-2\}$, in which $d'(n,k,i)=d(n,k,i)+1=n-1$. Thus, $d$ and $d'$ have essentially the same behavior, except from a very specific case.

\begin{proof}
	If $n\geq 2k+2i+2$, then $2k+2i\leq n-2$ and an optimal embedding is in a $(k+i)$-neighborly polytope in~$\RR^{2k+2i}$, whose vertices are in convex position.
	
	If $2k\leq n\leq 2k+2i+1$, then there is no $(k+i)$-neighborly embedding except for the isomorphism, and hence whether $d'(n,k,i)=n-2$ or $d'(n,k,i)=n-1$ depends on whether a non $(k-i-1)$-almost neighborly embedding exists. If $k=i+2$, then the embedding must be non $1$-almost neighborly, so the vertices cannot be in convex position, hence in this case $d'(n,k,i)=n-1$. Finally, if $k\geq i+3$, then the projection with point configuration $\Delta_{n-3}\oplus\Delta_{1}$ is $1$-almost neighborly, but not $2$-almost neighborly. Hence it has all its vertices in convex position and it is $i$-preserving.
	 
	The remaining cases follow by symmetry.
\end{proof}

\section{Relation with $k$-sets and $(i,j)$-partitions}\label{sec:ksets}

Let $S=\{s_1,\dots,s_n\}$ be a finite point set in~$\RR^d$. A subset of $S$ of cardinality $k$ is called a \defn{$k$-set} of $S$ if it is the intersection of~$S$ with an open halfspace. Studying the maximal possible number of $k$-sets is a central problem in combinatorial geometry, and only partial results are known. We refer to~\cite[Ch.~11]{MatousekLODG} for an introduction to the topic, to~\cite{Wagner2008} for a larger survey, and to~\cite{Sharir2011} for the latest improvement.

In~\cite{EdelsbrunnerValtrWelzl1997}, the $k$-set polytope $P_k(S)$ is defined as the convex hull of all $k$-barycenters of~$S$. Note that, if $\pi:\RR^n\to \RR^d$ is the linear projection with $\pi(\bfe_i)=\frac{1}{k}s_i$, then $P_k(S)=\pi(\Hnk)$ ($k$-set polytopes are defined without the $\frac{1}{k}$ factor in some references like~\cite{AndrzejakWelzl2003}). Hence, $k$-set polytopes for point sets of cardinality $n$ are projections of the $(n,k)$-hypersimplex. The importance of $P_k(S)$ lies in the fact that its vertices are in bijection with the $k$-sets of $S$.

In~\cite{AndrzejakWelzl2003}, Andrzejak and Welzl studied further the facial structure of $k$-set polytopes. To this end, they define an $(i,j)$-partition of $S$ as a pair $(A,B)$ of subsets of $S$ with $|A|=i$ and $|B|=j$ for which there is an oriented hyperplane $H$ such that $A=S\cap H$, and $B=S\cap H_{>0}$, where $H_{>0}$ is the positive open halfspace defined by~$H$. These generalize $k$-sets (which are $(0,k)$-partitions) and $j$-facets of point sets in general position (which are $(d,j)$-partitions).
Denoting by $D_{i,j}(S)$ the number of $(i,j)$-partitions of $S$, Andrzejak and Welzl observed that, for $S$ in general position,
\[f_{i-1}(P_k(S))=\begin{cases}1,& \text{ if }i=0,\\D_{0,k}(S),&\text{ if }i=1,\\\sum_{j=k-(i-1)}^{k-1} D_{i,j}(S)&\text{ otherwise};\end{cases}\]
which allowed them to use Euler's relation on $P_k(S)$ to derive linear relations on the numbers of $(i,j)$-partitions. See also~\cite[Sec.~3.2]{AndrzejakPhD}, which considers configurations that are not in general position.

In this section, we provide an interpretation of these concepts in terms of preserved faces under hypersimplex projections.

We define the \defn{dimension} of an $(i,j)$-partition $(A,B)$ as the dimension of the affine hull of~$A$.
Of course, if $S$ is in general position, then the dimension of any $(i,j)$-partition is $i-1$.

\begin{proposition}
Let $P_k(S)=\pi(\Hnk)\subset\RR^d$ be as before. Then
\begin{enumerate}[(i)]
 \item the vertices of $P_k(S)$ are in bijection with the $k$-sets of $S$, and
 \item the $e$-faces of $P_k(S)$ with $e\geq 1$ are in bijection with the $e$-dimensional $(i,j)$-partitions of~$S$ with $j+1\leq k\leq i+j-1$. 
\end{enumerate}
More precisely, for disjoint subsets $X,Y\subseteq [n]$ with $|Y|+1\leq k\leq |X|+|Y|-1$, 
and $A=\set{s_i}{i\in X}$ and $B=\set{s_j}{j\in Y}$, we have that
$(A,B)$ is an $e$-dimensional $(i,j)$-partition of $S$
if and only if the face 
$F= \Hnk^{I,J}$
is preserved under $\pi$, and $\pi(F)$ is $e$-dimensional; where $I=Y$ and $J=[n]\setminus(X\cup Y)$.

In particular, the strictly preserved $e$-faces with $e\geq 1$ are in bijection with the $(i-1)$-dimensional $(i,j)$-partitions of $S$ with $j+1\leq k\leq i+j-1$.

\end{proposition}
\begin{proof}
Faces of $P_k(S)$ are in bijection with faces of $\Hnk$ preserved under~$\pi$. More precisely, the face of $P_k(S)$ that maximizes the linear functional $f\in(\RR^d)^*$ is the image of the face~$F$ of $\Hnk$ maximized by $\pi^*(f)$, where $\pi^*$ denotes the adjoint of~$\pi$. That is, the vertices of~$F$ are the incidence vectors of the $S\in \binom{[n]}{k}$ that maximizes $\sum_{i\in S} f(s_i)$, where~$s_i=\pi(\bfe_i)$.

Fix $f\in(\RR^d)^*$ and let $e\in \binom{[n]}{k}$ be one of the subsets that maximizes $\sum_{i\in e} f(s_i)$. Set $c=\min_{i\in e} f(s_i)$, $X=\set{i\in [n]}{f(s_i)=c}$ and $Y=\set{i\in [n]}{f(s_i)>c}$. The hyperplane $f(x)=c$ defines an $(|X|,|Y|)$-partition $(A,B)$ with $A=\set{s_i}{i\in X}$ and $B=\set{s_j}{j\in Y}$.
Note that every $i\in Y$ belongs to $e$, and that for  $e'\in \binom{[n]}{k}$ we have $\sum_{i\in e'} f(s_i)=\sum_{i\in e} f(s_i)$ if and only if $Y\subseteq e'$ and $e'\subseteq  Y\cup X$. Therefore, the face of $\Hnk$ maximized by $\pi^*(f)$ is precisely $F=\Hnk^{I,J}$ with $I=Y$ and $J=[n]\setminus(X\cup Y)$. 

If $F$ is not a vertex, then we have $|Y|\leq k-1$ and $|X\cup Y|\geq k+1$. Conversely, every $(i,j)$-partition with $j+1\leq k\leq i+j-1$ corresponds to a linear functional that is maximized in such a face. Note that $\pi(F)$ is affinely equivalent to $P_{k-j}(A)$, which has the same dimension as the affine span of $A$ provided that $1\leq k-j\leq i-1$. 

The same argument works for vertices. The only subtlety lies in the fact that we claim that the preimage of a vertex of $P_k(S)$ must be a vertex of $\Hnk$. Indeed, while it is true that there might be subsets $e,e'\in \binom{[n]}{k}$ for which $\sum_{i\in e'} s_i=\sum_{i\in e} s_i$, such a subset cannot define a vertex of $P_k(S)$. The reason is that, if all the $s_i$'s are different, and if $f$ is a generic functional maximized at $\sum_{i\in e} f(s_i)$, then there must be some $j\in e'\setminus e$ with $f(s_j)>\frac{1}{k}\sum_{i\in e} f(s_i)$, and hence a subset $e''\in \binom{[n]}{k}$ with $\sum_{i\in e''} f(s_i)>\sum_{i\in e} f(s_i)$, contradicting the maximality of~$e$.
\end{proof}

\section{Hypergraphs with many barycenters in convex position}
\label{sec:extremal}

Now we study the extremal function $g_k(n,d)$ that counts the maximum number of barycenters in convex position that a $k$-uniform hypergraph on $n$ vertices embedded in~$\RR^d$ may have. As explained in Theorem~\ref{thm:extremal}, we distinguish three regimes: $d\geq 2k$, $k+1\leq d\leq 2k-1$ and $d\leq k$. 
By Theorem~\ref{thm:main}, we have $g_k(n,d)=\binom{n}{k}$ for $d\geq 2k$, which covers the first case.

By Theorem~\ref{thm:main}, we also know that $g_k(n,d)<\binom{n}{k}$ when $d\leq 2k-1$. By combining this result with de~Caen's bound on Tur\'an numbers for hypergraphs~\cite{DeCaen1983} we can get sharper upper bounds for $g_k$ when $d\leq 2k-1$, as $n$ grows.

Fix $k$ and  $1\leq d \leq 2k-1$. Using Theorem~\ref{thm:main}, we obtain that the maximum value $n=n_{k,d}$ so that $K_{n}^{(k)}$ has a convex embedding into $\RR^d$, for $d\geq 2$, is:

\begin{align}
	\label{aln:maxcomplete}
	n_{k,d}=\begin{cases}
		\floor{\frac{d}{2}}+k &\text{ if $1\leq d\leq 2k-4$},\\
                d+2 &\text{ if $d\in\{2k-3,2k-2,2k-1\}$}\\
                \infty &\text{ if $d\geq 2k$}.   
	\end{cases}
\end{align}
and for $d=1$, $n_{1,1}=2$,  $n_{k,1}=k$ for $k\geq 2$.

The first values of $n_{k,d}$ are contained in Table \ref{tab:maxvalues}.

\begin{table}[h]
\centering
\begin{tabular}{|c|cccccccccccccc|}
	\hline
$k\setminus d$ &     1    &     2    &     3    &     4    &     5    &     6     &     7     &     8     &     9     &    10     &    11     &    12     &    13     &    14     \\\hline
1              &\cellcolor{green}     2    & $\infty$ & $\infty$ & $\infty$ & $\infty$ & $\infty$  & $\infty$  & $\infty$  & $\infty$  & $\infty$  & $\infty$  & $\infty$  & $\infty$  & $\infty$  \\
2              &\cellcolor{green}     2    &\cellcolor{yellow}    4     &\cellcolor{yellow}     5    & $\infty$ & $\infty$ & $\infty$  & $\infty$  & $\infty$  & $\infty$  & $\infty$  & $\infty$  & $\infty$  & $\infty$  & $\infty$  \\
3              &\cellcolor{pink}     3    &\cellcolor{pink}    4     &\cellcolor{yellow}     5    &\cellcolor{yellow}     6    &\cellcolor{yellow}     7    & $\infty$  & $\infty$  & $\infty$  & $\infty$  & $\infty$  & $\infty$  & $\infty$  & $\infty$  & $\infty$  \\
4              &\cellcolor{pink}     4    &\cellcolor{pink}    5     &\cellcolor{pink}     5    &\cellcolor{pink}     6    &\cellcolor{yellow}     7    &\cellcolor{yellow}     8     &\cellcolor{yellow}     9     & $\infty$  & $\infty$  & $\infty$  & $\infty$  & $\infty$  & $\infty$  & $\infty$  \\
5              &\cellcolor{pink}     5    &\cellcolor{pink}    6     &\cellcolor{pink}     6    &\cellcolor{pink}     7    &\cellcolor{pink}     7    &\cellcolor{pink}     8     &\cellcolor{yellow}     9     &\cellcolor{yellow}    10     &\cellcolor{yellow}    11     & $\infty$  & $\infty$  & $\infty$  & $\infty$  & $\infty$  \\
6              &\cellcolor{pink}     6    &\cellcolor{pink}    7     &\cellcolor{pink}     7    &\cellcolor{pink}     8    &\cellcolor{pink}     8    &\cellcolor{pink}     9     &\cellcolor{pink}     9     &\cellcolor{pink}    10     &\cellcolor{yellow}    11     &\cellcolor{yellow}    12     &\cellcolor{yellow}     13    & $\infty$  & $\infty$  & $\infty$  \\
7              &\cellcolor{pink}     7    &\cellcolor{pink}    8     &\cellcolor{pink}     8    &\cellcolor{pink}     9    &\cellcolor{pink}     9    &\cellcolor{pink}    10     &\cellcolor{pink}    10     &\cellcolor{pink}    11     &\cellcolor{pink}    11     &\cellcolor{pink}    12     &\cellcolor{yellow}     13    &\cellcolor{yellow}     14    &\cellcolor{yellow}     15    & $\infty$  \\\hline
\end{tabular}
\caption{Values of $n_{k,d}$ for small values of $d$ and $k$. Yellow values are the cases $d\in\{2k-3,2k-2,2k-1\}$, red values are the cases $1\leq d\leq 2k-4$, and green values are the exceptional cases with $d=1$ where the standard formula does not hold.}\label{tab:maxvalues}

\end{table}
Let $EX(n,k,\ell)$ be the maximum number of hyperedges that a $k$-uniform hypergraph with $n$ vertices and no complete $\Kuh{\ell}{k}$ as an induced subhypergraph can have. These parameters are called \defn{Tur\'an numbers} for complete hypergraphs. We will use the following bound by de~Caen \cite{DeCaen1983}:

\begin{theorem}[\cite{DeCaen1983}]
\label{thm:decaen}
We have \[EX(n,k,\ell)\leq \left(1 - \frac{n-\ell+1}{n-k+1}\cdot \frac{1}{\binom{\ell-1}{k-1}}\right)\binom{n}{k}.\]

\end{theorem}

\begin{theorem}
\label{thm:upperbound}
	For $1\leq d\leq 2k-1$ we have $$g_k(n,d)\leq c_{k,d}\cdot n^k+o(n^k),$$ where 
	\[c_{k,d}=\frac{1}{k!}\left(1 - \frac{1}{\binom{n_{k,d}}{k-1}}\right), \] 
        for $n_{k,d}$ as defined in~\eqref{aln:maxcomplete}.
\end{theorem}
\begin{proof}
	If a $k$-uniform hypergraph $G$ has convex dimension $d$, then any induced subhypergraph must also have convex dimension $d$. In particular, its largest complete sub-hypergraph cannot have more than $n_{k,d}$ vertices. Therefore, we may apply de Caen's bound with $\ell=n_{k,d}+1$ to obtain that $G$ has at most $EX(n,k,n_{k,d}+1)$ edges. 
	
	The result follows by using that $\binom{n}{k}=\frac{n^k}{k!}+o(n^k)$ and collecting the $\frac{n-n_{k,d}}{n-k+1}$ coefficient in the $o(n^k)$ term. 
\end{proof}

For $d\geq k+1$, we have an accompanying lower bound for $g_k$ of the same asymptotic order. Denote by 
$\Kumh{m}{n}{k}=\Kuh{n,\dots,n}{k}$ the complete $m$-partite $k$-uniform hypergraph with $m$ parts of $n$~vertices each.
One obtains a first bound of size $g_k(n,d)\geq \left(\nicefrac{n}{k}\right)^k+ o(n^k)$ by showing that $\Kumh{k}{n}{k}$ has a convex embedding into $\RR^{k+1}$. This can be done using a particular case of a result by Matschke, Pfeifle and Pilaud~\cite{Matschke2011}. Namely, Theorem~2.6 in \cite{Matschke2011} (with parameters\footnote{We use $\rm k$ to differentiate their parameter from our variable $k$.} $\mathrm{r}=k$, $\mathrm{k}=0$ and $\mathrm{n_i}= n$ for each~$\mathrm{i}$) provides~$k$ sets of $n$ points in~$\RR^{k+1}$ whose Minkowski sum has all the possible $n^k$ vertices. Mapping the vertices of $\Kumh{k}{n}{k}$ to these sets gives the desired convex embedding.

For $k=2$ and $d=3$, this gives a lower bound of order $\nicefrac{n^2}{4}+o(n)$. However, for this case a better lower bound of size $\floor{\nicefrac{n^2}{3}}$ was found by Swanepoel and Valtr in~\cite[Theorem~6]{SwanepoelValtr2010} using a convex embedding of $K_{n,n,n}$ into $\RR^3$. 

We close the gap by providing below an improved lower bound of size \[g_k(n,d)\geq \binom{d}{k}\left(\frac{n}{d}\right)^k+ o(n^k)\] for any $d\geq k+1$. One can easily verify that this bound is increasing with $d$ by using Bernoulli's inequality. In particular, 
we have \[\binom{d}{k}\left(\frac{n}{d}\right)^k\geq \frac{n^k}{(k+1)^{(k-1)}} \geq \left(\frac{n}{k}\right)^k;\]
which shows that our bound improves the one arising from the construction in~\cite{Matschke2011}.
To do so, we construct a convex embedding of $\Kumh{d}{n}{k}$ into $\RR^d$ for any $d\geq k+1$.

\begin{theorem}\label{thm:lowerbound}
For fixed $d\geq k+1$, there is a convex embedding of the complete $d$-partite $k$-uniform hypergraph $\Kumh{d}{n}{k}$ into $\RR^{d}$. Therefore, $g_k(n,d)\geq \binom{d}{k}\left(\frac{n}{d}\right)^k+ o(n^k)$ as~$n\to \infty$. 
\end{theorem}
\begin{proof}
Let $\bfe_1,\dots, \bfe_{d}$ be the standard basis vectors of $\RR^d$, and set $\bfe_0=-\sum_{i=1}^{d-1} \bfe_i$ (notice that the sum starts at $1$ and ends at $d-1$).  
 
 Consider a set~$A$ of $n$ distinct positive real numbers. For any $k$-subset $S$ of $\{0,\dots,d-1\}$ we define the set $X_S$, of cardinality $n^k$, as 
 \[X_S:=\set{\sum_{i\in S} (a_i\cdot \bfe_i + a_i^2\cdot \bfe_d)}{a_i\in A\text{ for each }i\in S}.\]
 Let $X:=\bigcup_S X_S$ be the union of these point sets for all $k$-subsets of $\{0,\dots,d-1\}$. Note that $|X|=\binom{n}{d}n^k$, as the subsets are disjoint (here we use the positivity of~$A$). We will prove that the point set $X$ is in convex position. Notice that for any pair of $k$-subsets $S$ and $T$, there is a linear automorphism of~$X$ that sends $X_S$ to~$X_T$. Hence, it suffices to show that the points of $X_S$ are vertices of $\conv(X)$ for $S=\{1,\dots, k\}$.
 
 We do so by exhibiting a supporting hyperplane for each of these points. Fix a point $ \bfp=\sum_{i\in S} (a_i\cdot \bfe_i + a_i^2\cdot \bfe_d)\in X_S$, and consider the linear functional  $\bfv\in(\RR^d)^*$ given by $\bfv=-\bfe_d^* + \sum_{i\in S} 2a_i\cdot \bfe_i^*$. Then we have that $\sprod{\bfv}{\bfp}=\sum_{i\in S}  a_i^2$. We will see that 
 $\sprod{\bfv}{\bfq}<\sum_{i\in S}  a_i^2$ for any other $\bfq\in X$. Let $\bfq=\sum_{i\in T} (b_i\cdot \bfe_i + b_i^2\cdot \bfe_d)\in X_T$, and set $b_i=0$ for any $0\leq i\leq d-1$ not in $T$. 
 
 Then, using that $a_i,b_i\geq 0$ we see that
 \begin{align*}
  \sprod{\bfv}{\bfq}=\sum_{i\in S}(2a_ib_i - 2a_ib_0 -b_i^2) -\sum_{i\in T\ssm S}b_i^2\leq \sum_{i\in S}(a_i^2 -(a_i-b_i)^2)\leq \sum_{i\in S}a_i^2,
 \end{align*}
which can only be an equality if $S=T$ and $a_i=b_i$ for all $i\in S$; that is, if $\bfp=\bfq$.

This shows that all the points in $X_S$ are vertices of $\conv (X)$, and, by symmetry, that all the points of $X$ are vertices of $\conv (X)$.

To conclude the proof, let  $V=V_1\cup \cdots\cup V_d$ be the vertex set of $\Kumh{d}{n}{k}$, and consider a map $f:V\to \mathbb{R}^{d}$ that maps bijectively each~$V_i$ to $\set{a\cdot \bfe_{i-1}+a^2\cdot \bfe_{d}}{a\in A}$. 

Any hyperedge from $\Kumh{d}{n}{k}$ is obtained by choosing a $k$-subset $S$ of $\{0,\dots, d-1\}$ and then one vertex from each $V_i$ with $i\in S$. And hence every $k$-barycenter is precisely of the form $$\frac{1}{k}\sum_{i\in S} (a_i\cdot \bfe_i + a_i^2\cdot \bfe_d)$$ with $a_i\in A$ for $i\in S$. All these barycenters lie in convex position (they form $\frac{1}{k} X$), so $f$ is indeed a convex embedding into~$\RR^{d}$.
\end{proof}

Combining Theorems~\ref{thm:upperbound} and~\ref{thm:lowerbound}, we get the following estimation for the coefficient $\gamma_{k,d}$ of $n^k$ in the asymptotic development of $g_k(n,d)$ in the range $k+1\leq d\leq 2k-1$. 

\begin{cor}
\label{cor:asymptotic_constant}
	If $k+1\leq d\leq 2k-1$, then $g_k(n,d)$ is in $\Theta(n^k)$. More precisely, $$g_k(n,d)=\gamma_{k,d}\cdot n^k+o(n^k)$$ for some constant $\gamma_{k,d}$ satisfying 
 
 \[\binom{d}{k}\frac{1}{d^k} \leq \gamma_{k,d} \leq \frac{1}{k!}\left(1 - \frac{1}{\binom{n_{k,d}}{k-1}}\right)\]
 where 
 \[	n_{k,d}=\begin{cases}
		d+2 &\text{ if $d\geq 2k-3$, }\\
		\floor{\frac{d}{2}}+k &\text{ if $1\leq d\leq 2k-4$}.
	\end{cases}
	\]
 for $d\neq 1$, 
 and $n_{1,1}=2$ and $n_{k,1}=k$ for $k\geq 2$.
\end{cor}

For $d\leq k$, we do not know the asymptotic order of $g_k(n,d)$. We will show that $g_k(n,d)\in \Omega(n^{d-1})$, but we do not know any non-trivial upper bound. 
\begin{remark}\label{rmk:upperboundk>=k}
There is an obvious $O(n^d)$ bound for the number of $k$-sets of a configuration of $n$ points in~$\RR^d$, as every $k$-set is separated from the other points by an affine hyperplane, and the number of separating hyperplanes is at most $O(n^d)$ (see for example~\cite[Thm~2.2]{AndrzejakWelzl2003}). In the published version of this paper~\cite{jctb2021}, we wrongly claimed that this argument extended to arbitary hypergraphs and gave an upper bound of order~$O(n^d)$ for $g_k(n,d)$. We thank Brett Leroux for noticing our mistake and letting us know. We only have the trivial upper bound~$\binom{n}{k}$ for $g_k(n,d)$. The problem of finding an upper bound that is independent from~$k$ remains open.
\end{remark}

We conclude with a construction of Weibel~\cite{Weibel2012} for Minkowski sums of polytopes that will provide a lower bound of $\Omega(n^{d-1})$ for $g_k(n,d)$. (We refer to~\cite[Sect.~5]{Weibel2012} for the details of the construction.) 

\begin{theorem}[{Theorem~3 in~\cite{Weibel2012}}]\label{thm:weibel}
Let $k\geq d$ and $n\geq d+1$, then there exist $d$-polytopes $P_1, \dots , P_k\subset \RR^d$ with $n$ vertices such that 
$P_1+\cdots+P_k$ has $\Theta(n^{d-1})$ vertices.
\end{theorem}

\begin{cor}
 For $k\geq d$, there is a subhypergraph of the complete $k$-uniform $k$-partite hypergraph $\Kuh{k\times n}{k}$ with $\Theta(n^{d-1})$ hyperedges that has a convex embedding into $\RR^{d}$. Therefore, we have that $g_k(n,d)$ is in~$\Omega(n^{d-1})$ as~$n\to \infty$.
\end{cor}
\begin{proof}
Consider the polytopes $P_i\subset \RR^d$, $1\leq i\leq k$,  with vertices $\{p_{i1},\dots,p_{in}\}$ from Theorem~\ref{thm:weibel}. We define the subhypergraph $H$ of~$\Kuh{n,n,\ldots,n}{k}$
with vertex set $V=V_1\cup \cdots\cup V_k$ with \[V_i=\{v_{i1},\ldots, v_{in}\} \text{ for $i\in [k]$}\]
and whose hyperedges correspond to the $k$-tuples $(v_{1j_1},\dots, v_{1j_k})$ such that $p_{1j_1}+\cdots +p_{kj_k}$ is a vertex of $P_1+\cdots+P_k$. Then the map that sends $v_{ij}$ to $p_{ij}$ is a convex embedding.
\end{proof}

\section{Discussion and open problems}\label{sec:open}

\subsubsection*{Other hypergraphs:}
After studying complete hypergraphs, it would be interesting to determine the convex dimension of other families of uniform hypergraphs. Note that our approach via Lemmas~\ref{lem:equivalence} and~\ref{lem:projection_lemma} also extends to other hypergraphs when the hypersimplex is replaced by the corresponding hyperedge polytopes.

A particularly interesting family of uniform hypergraphs that comes to mind are (sets of bases of) \defn{matroids}. Their hyperedge polytopes, known as \defn{matroid polytopes}, have been extensively studied and many of their properties are well understood. They are in particular a relevant family in the context of the convex combinatorial optimization problems that originally motivated the study of the convex dimension of hypergraphs~\cite{OnnRothblum2004}. The associated optimization problem is known as \defn{convex matroid optimization}~\cite{Onn2003}.

\subsubsection*{Asymptotic behavior of $g_k(n,d)$:} It is a challenging question to understand the asymptotic growth of $g_k(n,d)$ when $d\leq k$. For $d\leq k$, we only know that $g_k(n,d)\in \Omega(n^{d-1})$. We originally claimed an $O(n^d)$ upper bound, but our argument was wrong. It would be interesting to find an upper bound for $g_k(n,d)$ that does not depend on~$k$, but any upper bound improving on the trivial $\binom{n}{k}$ would already be interesting progress. When $d=k=2$, it is known that $g_2(n,2)\in \Theta(n^{\frac{4}{3}})$~\cite{Bilka2010,EisenbrandPachRothvossSopher2008}.

When $k+1\leq d\leq 2k-1$, we know that $g_k(n,d)=\gamma_{d,k}\cdot n^k+o(n^k)$, but the exact value of $\gamma_{d,k}$ is unknown (see Theorem~\ref{thm:extremal}), even when $k=2$ and $d=3$~\cite{SwanepoelValtr2010}.

\subsubsection*{Large subsets of Minkowski sums:} Restricting this extremal problem to subgraphs of complete $k$-partite $k$-uniform hypergraphs is equivalent to the question of finding large convexly independent subsets of the Minkowski sum of $k$ point sets. Indeed, the $k$-barycenters of an embedding of a complete $k$-partite $k$-uniform hypergraph are (a dilation) of the Minkowski sum of the embeddings of each of the $k$ parts.
The planar case with $k=2$ has been an active area of research during the last decade. The unconstrained version was (asymptotically) solved in~\cite{Bilka2010,EisenbrandPachRothvossSopher2008}, and the case where the point sets are themselves in convex position was (asymptotically) solved in~\cite{Tiwary2014,SkomraThomasse}. Some of these cases were considered in~\cite{GarciaMarcoKnauer2016}, who introduced variants with weak convexity and sharpened the bounds when the two point sets coincide and are in convex position.
The case $k=2$ and $d=3$ was studied in~\cite{SwanepoelValtr2010}. However, we are not aware of any result for larger~$k$ except for the lower bounds arising from Minkowski sums with many vertices.
Finding the maximal number of vertices of a Minkowski sum has been solved~\cite{AdiprasitoSanyal2016} (although closed formulas seem rather involved and are not explicit). 

\subsubsection*{Extremal problem for strong convex embeddings:} The notion of strong convex dimension of a hypergraph poses an analogous extremal problem on the maximum number of hyperedges \defn{$h_k(n,d)$} that a $k$-uniform hypergraph on $n$ vertices with a strong convex embedding to $\RR^d$ can have. The asymptotic values of $g_k(n,d)$ and $h_k(n,d)$ may largely differ: it is shown in \cite{HalmanOnnRothblum2007, GarciaMarcoKnauer2016} that $h_2(n,2)$ is linear, while $g_2(n,2)$ is in $\Theta(n^{4/3})$ \cite{Bilka2010,EisenbrandPachRothvossSopher2008}. What are other quantitative and qualitative differences between convex and strong convex embeddings when $k>2$?

\subsubsection*{Combinatorial and topological hypersimplices:} 
The (topological) van Kampen-Flores Theorem~\cite{Flores1935,vanKampen1932} states that the $i$-skeleton of the $(2i+2)$-simplex cannot be embedded in~$\RR^{2i}$. This begs the question whether an analogous result for all hypersimplices also holds: Is $d(n,k,i)-1$ the smallest dimension where the $i$-skeleton of~$\Hnk$ can be topologically embedded\footnote{The shift by~$1$ comes from the fact that in a polytopal projection we are embedding it into the boundary of the $d$-polytope, which is homeomorphic to a $(d-1)$-sphere.}?

This would imply that Theorem~\ref{thm:skeleton} also holds for combinatorial hypersimplices. 
This is not immediate from our results, since there are plenty of polytopes that are combinatorially but not affinely isomorphic to a hypersimplex (in fact, the realization spaces of hypersimplices are far from being understood~\cite{GrandePadrolSanyal2018}). 

\section*{Acknowledgements} 

We are grateful to Kolja Knauer for introducing us to this problem and many interesting discussions. We thank Raman Sanyal for useful comments on an earlier version of this manuscript. We also want to thank
Edgardo Rold\'an-Pensado for keeping us informed about their progress on convex embeddings of multipartite graphs into~$\RR^3$. We want to thank anonymous reviewers for their valuable comments, one of which pointed out an additional reference and an alternative proof of Theorem \ref{thm:main}.

Finally, we are very grateful to Brett Leroux, who warned us about a mistake in the proof of one of our originally claimed results~\cite[Thm.~6.5]{jctb2021}.

\bibliographystyle{amsplain}
\bibliography{midpoints}

\end{document}